\documentclass[a4paper,11pt,oneside]{article}
\usepackage[latin1]{inputenc}
\usepackage{amssymb, amsmath}
\usepackage{moreverb}
\usepackage{vmargin}
\usepackage{fancyhdr}
\usepackage{amsfonts}
\usepackage{bm}
\usepackage{fancybox}
\usepackage{amsthm}
\usepackage{amsmath}

\newtheorem{theorem}{Theorem}[section]

\newtheorem{lemma}[theorem]{Lemma}

\newtheorem{definition}[theorem]{Definition}

\newtheorem{remark}[theorem]{Remark}

\usepackage[
  linktocpage,
  urlbordercolor={1 0.5 0.125},
  filebordercolor={1 0.5 0.125},
  linkbordercolor={0.25 1 0.25},
  citebordercolor={0.25 1 0.25},
  pdfborderstyle={/S/U/W 1}, % underline links in Acrobat
  breaklinks]{hyperref}  % for hyperlinks
  
\begin{document}
\ \begin{center}
\begin{LARGE}
\bf {Well-posedness for a model\\
 of individual clustering} \\
\end{LARGE}
\vspace{0.5cm}
Elissar Nasreddine\\
\vspace{0.2cm}
 \begin{small}
\textit{Institut de Math\'ematiques de Toulouse, Universit\'e de Toulouse,}\\
 \textit{F--31062 Toulouse cedex 9, France}\\
 \vspace{0.1cm}
 e-mail: elissar.nasreddine@math.univ-toulouse.fr\\
 \today
\end{small}
\end{center}
\textbf{Abstract}. We study the well-posedness of a model of individual clustering. Given \\$p>N\geq 1$ and an initial condition in $W^{1,p}(\Omega)$, the local existence and uniqueness of a strong solution is proved. We next consider two specific reproduction rates and show global existence if $N=1$, as well as, the convergence to steady states for one of these rates.
\vspace{0.5cm}
\section {Introduction}
In \cite{models}, a model for the dispersal of individuals with an additional aggregation mechanism is proposed. More precisely, classical models for the spatial dispersion of biological populations read
\begin{equation}\label{intro}
\partial_t u=\Delta (\Phi(u))+f(u,t,x) .
\end{equation}
where $u(t,x)$ denotes the population density at location $x$ and time $t$, and $f(u, t, x)$ represents the population supply, due to births and deaths. The dispersal of individuals is either due to random motion with $\Phi(u)=u$ or rests on the assumption that individuals disperse to avoid crowding and $\Phi$ satisfies 
\begin{equation}
\Phi(0)=0,\ \ \mathrm{and}\ 
\Phi'(u)>0,\ \ \mathrm{for} \ u>0.\label{condition}
\end{equation}
No aggregation mechanism is present in this model though, as discussed in \cite{models}, the onset of clustering of individuals in a low density region might balance the death and birth rates and guarantee the survival of the colony. To account for such a phenomenon, a modification of the population balance \eqref{intro} is proposed in \cite{models} and reads
\begin{equation}\label{aaa}
\partial_tu=-\nabla\cdot(u\ \bm{V}(u,t,x))+u\ E(u,t, x).
\end{equation}
where $\bm{V}$ is the average velocity of individuals, and $E$ is the net rate of reproduction per individual at location $x$ and time $t$. To complete the model, we must specify how $\bm{V}$ is related to $u$ and $E$. Following \cite{models}, we assume that each individual disperses randomly with probability $\delta\in (0,1)$ and disperses deterministically with an average velocity $\bm{\omega}$ so as to increase his expected rate of reproduction with probability $1-\delta$. The former is accounted for by a usual Fickian diffusion $\frac{\nabla u}{u}$ while the latter should be in the direction of increasing $E(u,t,x)$, say, of the form $\lambda\ \nabla E(u,t,x)$ with $\lambda>0$. A slightly different choice is made in \cite{models} and results in the following system 
\begin{equation}
\label{sys}
 \begin{array}{lcl}
 \partial_tu&=&\delta \ \Delta u-(1-\delta)\ \nabla\cdot (u\ \bm{\omega})+u\ E(u,t,x)\\
 -\varepsilon\ \Delta \bm{\omega}+\bm{\omega}&=&\lambda \nabla E(u,t,x).
 \end{array}
 \end{equation}
 After a suitable rescaling, and assuming that the environment is homogeneous, \eqref{sys} becomes
 \begin{equation}
 \label{sys1}
 \begin{array}{lcl}
 \partial_tu&=&\delta \ \Delta u-\nabla\cdot (u\ \bm{\omega})+r\ u\ E(u)\\
 -\varepsilon\ \Delta \bm{\omega}+\bm{\omega}&=& E'(u)\ \nabla u,
 \end{array}
 \end{equation}
 for $x\in \Omega$ and $t\geq 0$, where $\Omega$ is an open bounded domain of $\mathbb{R}^N$, $1\leq N\leq3$. We supplement \eqref{sys1} with no-flux boundary conditions
 \begin{equation}\label{cng}
 \bm{n}\cdot \nabla u=\bm{n}\cdot \bm{\omega}=0\ \  x\in \partial \Omega, t\geq0,
 \end{equation}
as suggested in \cite{models}. However, the previous boundary conditions \eqref{cng} are not sufficient for the well-posedness of the elliptic system verified  by $\bm{\omega}$ in several space dimensions and we must impose the following additional condition given in \cite{asimplified, unprobleme,quelques}:
\begin{equation}\label{bord}
\displaystyle \partial_n \bm{\omega}\times n=0, \ \ x\in \partial \Omega, t>0.
\end{equation}
As usual, $v\times \omega$ is the number $v_1\ \omega_1+v_2\ \omega_2$ if $N=2$ and the vector field
 $(v_2\ \omega_3-v_3\ \omega_2, -v_1\ \omega_3+v_3\ \omega_1,v_1\ \omega_2-v_2\ \omega_1)$ if $N=3$. We note that the boundary condition \eqref{bord} is useless if $N=1$.\\
 
 Summarizing, given a sufficiently smooth function $E$, parameters $\delta>0$, $\varepsilon \geq0$ and $r\geq 0$, our aim in this paper is to look for $(u,\bm{\omega})$ solving the problem
\begin{equation}
\label{grin}
\left\{
\begin{array}{llll}
\displaystyle \partial_t u&=& \delta\ \Delta u-\nabla\cdot(u\ \bm{\omega})+r\ u\ E(u),& x\in \Omega, t>0 \\
\displaystyle -\varepsilon\ \Delta \bm{\omega}+\bm{\omega}&=& \nabla E(u),& x\in \Omega, t>0 \\
\displaystyle \partial_nu=0&,& \bm{\omega}\cdot n=0,& x\in \partial \Omega, t>0\\
\displaystyle \partial_n \bm{\omega}\times n&=&0,&  x\in \partial \Omega, t>0\\
\displaystyle u(0,x)&=&u_0(x),& x\in \Omega.
\end{array}
\right.
\end{equation}
In the first part of this paper, we show that, for $p>N$, the system \eqref{grin} has a maximal solution $u$ in the sense of Definition \ref{def} where $u\in C\left( [0, T_{\mathrm{max}}), W^{1,p}(\Omega)\right)\cap C\left( (0, T_{\mathrm{max}}), W^{2,p}(\Omega)\right) $, see Theorem \ref{local}.\\

In the second part, we turn to the global existence issue and focus on space dimension 1, and two specific forms of $E$ suggested in \cite{models}: the ``bistable case" where
 $E(u)=(1-u)(u-a)$ for some $a\in (0,1)$, see Theorem \ref{the1}, and the ``monostable case"  $E(u)=1-u$.
In both cases, we prove the global existence of solution. In addition, in the monostable case, i.e $E(u)=1-u$, thanks to the Liapunov functional
 $$L(u)=\int_{-1}^1 ( u\ \log u-u+1)\ dx,$$
 we can study the asymptotic behaviour of solutions for $t$ large, and show that the solution $u$ converges, when $t$ goes to $\infty$, to a steady state in $L^2(-1,1)$,
 see Theorem \ref{theo2}.\\
 
 In the third part, we investigate the limiting behaviour as $\varepsilon \rightarrow 0$. Heuristically, when $\varepsilon$ goes to zero, the velocity $\bm{\omega}$ becomes sensitive to extremely local fluctuations in $E(u)$, and the system \eqref{grin} reduces to the single equation
 \begin{equation}\label{paraboliq}
 \partial_tu=\nabla\cdot \left((\delta-u\ E'(u))\ \nabla u\right)+r\ u\ E(u).
 \end{equation} 
 Clearly \eqref{paraboliq} is parabolic only if $\delta-u\ E'(u)\geq 0$ for all $u>0$. This is in particular the case when $E(u)=1-u$, see Theorem \ref{limitep}. But this limit is not well-posed in general. As a result the population distribution may become discontinuous when neighbouring individuals decide to disperse in opposite direction, that is in particular the case when $E(u)=(1-u)(u-a)$.\\

\section{Main results}
Throughout this paper and unless otherwise stated, we assume that 
$$E\in C^2(\mathbb{R}),\ \delta>0,\ \varepsilon>0,\ r\geq0.$$
We first define the notion of solution to \eqref{grin} to be used in this paper.
\begin{definition}\label{def}
Let $T>0$, $p> N$,  and an initial condition $u_0\in W^{1, p}(\Omega)$ . A strong solution to \eqref{grin} on $[0, T)$ is a function
$$u \in C \left( [0,T), W^{1,p}(\Omega)\right)\cap C\left( (0,T), W^{2,p}(\Omega)\right),$$ such that
\begin{equation}
\label{eqdeu}
\left\{
\begin{array}{llll}
\displaystyle\partial_t u&=& \delta\ \Delta u-\nabla\cdot(u\ \bm{\omega}_u)+r\ u\ E(u),& \mathrm{a.e.\ in }\ [0,T)\times\Omega\\
\displaystyle u(0,x)&=&u_0(x),& \mathrm{a.e.\ in }\ \Omega\\
\displaystyle \partial_nu&=&0,& \mathrm{a.e.\ on }\ [0,T)\times\partial\Omega,
\end{array}
\right.
\end{equation}
where, for all $t\in [0,T)$, $\bm{\omega}_u(t)$ is the unique solution in $W^{2,p}(\Omega)$ of
\begin{equation}
\label{eqdefi}
\left\{
\begin{array}{llll}
\displaystyle-\varepsilon\Delta \bm{\omega}_u(t)+\bm{\omega}_u(t)&=& \nabla E(u(t))& \mathrm{a.e.\ in}\ \Omega\\
\displaystyle\bm{\omega}_u(t)\cdot n=\partial_n\bm{\omega}_u(t) \times n&=&0& \mathrm{a.e.\ on}\ \partial \Omega
\end{array}
\right.
\end{equation}
\end{definition}
Our first result gives the existence and uniqueness of a maximal solution of \eqref{grin} in the sense of Definition \ref{def}.
\begin{theorem}\label{local}
Let $p>N$ and a nonnegative function $u_0\in W^{1,p}(\Omega)$. Then there is a unique maximal solution $u\in C\left( [0, T_{\mathrm{max}}), W^{1,p}(\Omega)\right)\cap C\left( (0, T_{\mathrm{max}}), W^{2,p}(\Omega)\right) $ to \eqref{grin} in the sense of Definition \ref{def}, for some $T_{\mathrm{max}}\in (0,\infty]$. In addition, $u$ is nonnegative.\\
Moreover, if for each $T>0$, there is $C(T)$ such that
$$||u(t)||_{W^{1, p}}\leq C(T),\ \ \mathrm{for\ all}\ t\in[0,T]\cap[0,T_{\mathrm{max}}),$$
then $T_{\mathrm{max}}=\infty$.
\end{theorem}
The proof of the previous theorem relies on a contraction mapping argument.\\

We now turn to the global existence issue and focus on the one dimensional case, where $E(u)$ has the structure suggested in \cite{models}.\\
In the following theorem we give the global existence of solution to \eqref{grin} in the bistable case, that is when $E(u)=(1-u)(u-a)$, for some $a\in (0,1)$.
\begin{theorem} \label{the1}Assume that $u_0$ is a nonnegative function in $W^{1,2}(-1,1)$,\\ and $E(u)=(1-u)(u-a)$ for some $a\in (0,1)$. Then \eqref{grin} has a global nonnegative solution $u$ in the sense of Definition \ref{def}.
\end{theorem}
The proof relies on a suitable cancellation of the coupling terms in the two equations which gives an estimate for $u$ in $L^\infty (L^2)$ and for $\bm{\omega}$ in $L^2\left( W^{1,2}\right)$.\\

Next, we can prove the global existence of a solution to \eqref{grin} in the monostable case, that is, when $E(u)=1-u$, and we show that the solution converges as $t\rightarrow\infty$ to a steady state. More precisely, we have the following theorem
\begin{theorem}\label{theo2}Assume that $u_0$ is a nonnegative function in $W^{1,2}(-1,1)$,\\
  and $E(u)=1-u$. There exists a global nonnegative solution $u$ of \eqref{grin} in the sense of Definition \ref{def} which belongs to $L^\infty\left([0, \infty); W^{1, 2}(-1,1)\right)$.\\
 In addition, if $r=0$, $$\lim\limits_{t \to \infty}\left\lvert\left\lvert u(t)-\frac{1}{2}\int_{-1}^1 u_0\ dx \right\rvert\right\rvert_2=0,$$ 
 and if $r>0$ the solution $u(t)$ converges either to $0$ or to $1$ in $L^2(-1,1)$ as $t\rightarrow \infty$.
\end{theorem}
In contrast to the bistable case, it does not seem to be possible to begin the global existence proof with a $L^\infty(L^2)$ estimate on $u$. Nevertheless, there is still a cancellation between the two equations which actually gives us an $L^\infty(L \log L)$ bound on $u$ and a $L^2$ bound on $\partial_x\sqrt u$.
\begin{remark} \label{remark}We note that when $N=1$, there is a relation between our model when\\
 $E(u)=1-u$ and $r=0$, and the following chemorepulsion model studied in \cite{global}
\begin{equation}
\label{rep}
\left\{
\begin{array}{llll}
\displaystyle \partial_tu&=& \delta\ \partial_{xx}^2 u+\partial_x(u\ \partial_x \psi),&\ \mathrm{in}\ (0,\infty)\times (-1,1)\\
\displaystyle -\varepsilon\ \partial_{xx}^2\psi+\psi&=&u,& \ \mathrm{in} \ (0,\infty)\times (-1,1)\\
\partial_xu(t,\pm 1)&=& \partial_x \psi(t,\pm 1)=0& \ \mathrm{on}\ (0, \infty).
\end{array}
\right.
\end{equation}
Indeed, define $\varphi=-\partial_x \psi,$ and substitute it into \eqref{rep}. Then differentiating the second equation in \eqref{rep} we find
\begin{equation}
\left\{
\begin{array}{llll}
\displaystyle \partial_tu&=& \delta\ \partial_{xx}^2 u-\partial_x(u\ \varphi),&\ \mathrm{in}\ (0,\infty)\times (-1,1)\\
\displaystyle -\varepsilon\ \partial_{xx}^2\varphi+\varphi&=&-\partial_xu=\partial_x E(u),& \ \mathrm{in}\ (0,\infty)\times (-1,1)\\
\partial_xu(t,\pm 1)&=& \varphi(t,\pm 1)=0&\ \mathrm{on}\ (0, \infty).
\end{array}
\right.
\end{equation}
So that $u$ is a solution to our model.
\end{remark}

When $E(u)=1-u$, the limit $\varepsilon\rightarrow 0$ is formally justified and \eqref{grin} takes the qualitative form of \eqref{intro} with $\Phi(u)= \delta u+\frac{1}{2}\ u^2$. In this example though, since $E'<0$, the individuals dispersing so as to maximise $E$ would seek isolation, and there is clearly no mechanism capable of producing aggregation of individuals. This observation is actually consistent with Remark \ref{remark}. 

\begin{theorem}\label{limitep}
Assume that $u_0$ is a nonnegative function in $W^{1,2}(-1,1)$,\\
 and that $E(u)=1-u$. For $\varepsilon>0$ let $u_\varepsilon$ be the global solution to \eqref{grin} given by\\ Theorem \ref{theo2}. Then, for all $T>0$, 
\begin{equation}\label{l2}\lim\limits_{\varepsilon \to 0} \int_0^T ||u_{\varepsilon}(t)- u(t)||_2^2\  dt=0,\end{equation}
where $u$ is the unique solution to
\begin{equation}\label{limite}
\left\{
\begin{array}{llll}
\partial_tu&=&  \partial_{xx}^2 (\delta\ u+ \frac{1}{2}\ u^2)+r\ u\ (1-u),& x\in (-1,1), t>0,\\
u(0,x)&=&u_0(x),& x\in (-1,1),\\
\partial_xu(t, \pm 1)&=&0,& t>0.
\end{array}
\right.
\end{equation}
\end{theorem}
Since $\delta +u>0$ for $u\geq0$, the previous equation \eqref{limite} is uniformly parabolic and has a unique solution $u$, see \cite{linear} for instance.\\
The proof of Theorem \ref{limitep} is performed by a compactness method.
\section{Preliminaries}
We first recall some properties of the following system,
\begin{equation}
\label{ellip}
\left\{
\begin{array}{llll}
\displaystyle -\varepsilon\ \Delta \bm{\omega}+\bm{ \omega}&=&f,& \mathrm{in}\ \Omega, \\
\displaystyle \bm{\omega}\cdot n&=&0,& \mathrm{on}\ \partial \Omega,\\
\displaystyle \partial_n \bm{\omega}\times n&=&0,& \mathrm{on}\ \partial \Omega,
\end{array}
\right.
\end{equation}
where $f\in (L^p(\Omega))^N$ and $\Omega$  is a bounded open subset of $\mathbb{R}^N$, $N=2, 3$. Let us first consider weak solutions of \eqref{ellip}. For that purpose, we define
$$W_1=\lbrace \mathbf{v} \in (H^2(\Omega))^N; \mathbf{v}\cdot n=0,\ \mathrm{and}\ \partial_n \mathbf{v}\times n=0 \ \mathrm{on}\ \partial \Omega\rbrace$$
and take $W$ as the closure of $W_1$ in $(H^1(\Omega))^N$.\\
If $f\in (L^2(\Omega))^N$, the weak formulation for \eqref{ellip} is
\begin{equation}
\label{faible}
\left\{
\begin{array}{ccc}
\displaystyle \varepsilon\ \int_\Omega \nabla  \bm{\omega}\cdot \nabla \mathbf{v}+\int_\Omega  \bm{\omega}\cdot \mathbf{v} & =\displaystyle\int_\Omega f\cdot  \mathbf{v},\ &\mathrm{for\ all}\ \mathbf{v}\in W  \\
\displaystyle \bm{\omega}\in  W&&
\end{array}
\right.
\end{equation}
\\
where $\nabla \bm{\omega}\cdot \nabla \mathbf{v}=\sum_i \nabla \omega_i \cdot \nabla v_i $ and $\bm{\omega}\cdot \mathbf{v}=\sum_i \omega_i\ v_i$.\\

We recall some results about the existence, regularity and uniqueness of solution for \eqref{faible}, see \cite{asimplified, unprobleme}.
\begin{theorem}
If $f\in (L^2(\Omega))^N$, \eqref{faible} has a unique solution in $W$ and there is $C=C(\Omega, N)$ such that
$$||\bm{\omega}||_{W}\leq C\ ||f||_{2}.$$
\end{theorem}
We next consider strong solutions of \eqref{ellip}, that is, solutions solving \eqref{ellip} a.e. in $\Omega$. In this direction the existence and uniqueness of the strong solution to \eqref{ellip} is proved in \cite{quelques}:
\begin{theorem}\label{reg}
If $f\in (L^p(\Omega))^N$ with $1<p<\infty$, \eqref{ellip} has a unique solution in $(W^{2, p}(\Omega))^N$ and 
\begin{equation}\label{regularite}||\bm{\omega}||_{W^{2, p}}\leq \frac{K(p)}{\varepsilon} \ ||f||_{p},\end{equation}
where $K(p)=K(p,\Omega, N)$.
\end{theorem}
In other words, the strong solution has the same regularity as elliptic equations with classical boundary conditions.\\

We finally recall some functional inequalities: in several places we shall need the following version of Poincar\'e's inequality
 \begin{equation}\label{poincare}
 ||u||_{W^{1,p}}\leq C\ \left(||\nabla u||_p+||u||_q\right),\ \ u\in W^{1,p}(\Omega)
 \end{equation}
with arbitrary $p\geq1$ and $q\in [1, p]$. Also, we will frequently use the Gagliardo-Nirenberg inequality
\begin{equation}
\label{gagliardo}
||u||_p\leq C\ ||u||_{W^{1,2}}^\theta\ ||u||_q^{1-\theta},\ \ \mathrm{with}\ \theta=\frac{\frac{N}{q}-\frac{N}{p}}{1-\frac{N}{2}+\frac{N}{q}}, \ u\in W^{1,2}(\Omega)
\end{equation}
which holds for all $p\geq 1$ satisfying \ $p\ (N-2)<2\ N$ and $q\in [1,p)$.
\section{Local well-posedness}
Throughout this section, we assume that 
\begin{equation}
E\in C^2(\mathbb{R}),\ \mathrm{and\ set}\ \tilde{E}(z)=z\ E(z)\ \mathrm{for}\ z\in \mathbb{R}.
\end{equation}
\begin{proof}[Proof of Theorem \ref{local}]
We fix $p>N$, $R>0$, and define for $T\in (0,1)$ the set 
$$X_R(T):= \left\{ u\in C\left( [0,T]; W^{1,p}(\Omega)\right), \sup_{t\in [0,T]}||u(t)||_{W^{1,p}}\leq R\right\},$$
which is a complete metric space for the distance
$$d_X(u,v)=\sup_{t\in [0,T]}||u(t)-v(t)||_{W^{1,p}},\ \ (u,v)\in X_R(T)\times X_R(T).$$
For $u\in X_R(T)$,  and $t\in [0,T]$, the embedding of $W^{1,p}(\Omega)$ in $L^\infty(\Omega)$ ensures that $\nabla E(u(t))\in L^p(\Omega)$ so that \eqref{ellip} with $f=\nabla E(u)$ has a unique solution $\bm{\omega}_u\in \left(W^{2,p}(\Omega)\right)^N$. We then define $\Lambda(u)$ by
\begin{equation}
\Lambda u(t,x)= (e^{t\ (\delta \ \Delta)}\ u_0)(x)+\int_0^t e^{(t-s)\ (\delta \ \Delta)}\left[ -\nabla\cdot(u\ \bm{\omega}_u)+r \ \tilde{E}(u)\right](s,x)\ ds,
\end{equation}
for $(t,x)\in [0,T]\times \Omega$, where  $\left(e^{t\ (\delta\ \Delta)}\right)$ denotes the semigroup generated in $L^p(\Omega)$ by $\delta \ \Delta$ with homogeneous Neumann boundary conditions. We now aim at showing that $\Lambda$ maps $X_R(T)$ into itself, and is a strict contraction for $T$ small enough. In the following, $(C_i)_{i\geq 1}$ and $C$ denote positive constants depending only on $\Omega$, $\delta$, $r$, $\varepsilon$, $E$, $p$ and $R$.
\begin{itemize}
\item Step 1. $\Lambda$ maps $X_R(T)$ into itself.\\

We first recall that there is $C_1>0$ such that
\begin{equation}\label{vinf}
||v||_\infty\leq C_1 ||v||_{W^{1,p}},
\end{equation}
and
\begin{equation}\label{semigr}
||e^{t(\delta\ \Delta)}\ v||_{W^{1,p}}\leq C_1\ ||v||_{W^{1,p}},\ \mathrm{and}\ ||\nabla e^{t\ (\delta \ \Delta)}\ v||_p\leq C_1\ \delta^{-\frac{1}{2}}\ t^{-\frac{1}{2}}\ ||v||_p,
\end{equation}
for all $v\in W^{1,p}(\Omega)$. Indeed, \eqref{vinf} follows from the continuous embedding of $W^{1,p}(\Omega)$ in $L^\infty(\Omega)$ due to $p>N$ while \eqref{semigr} is a consequence of the regularity properties of the heat semigroup.\\
Consider $u\in X_R(T)$, and $t\in [0, T]$. It follows from \eqref{semigr} that
\begin{eqnarray}
||\Lambda u(t)||_p&\leq& C_1\ ||u_0||_p+\int_0^t ||\nabla e^{(t-s)(\delta\ \Delta)}\ (u\ \bm{\omega}_u)(s)\ ||_p \ ds\nonumber\\
&+&r\int _0^t ||e^{(t-s)(\delta \ \Delta)}\ \tilde{E}(u)(s)||_p\ ds\nonumber\\
&\leq& C_1\ ||u_0||_p+C_1\ \delta^{-\frac{1}{2}}\int^t_0(t-s)^{-\frac{1}{2}}\  ||u\ \bm{\omega}_u(s)||_p \ ds\nonumber\\
&+&C_1\ r\ \int^t_0 || \tilde {E}(u)(s)||_p \ ds.\nonumber
\end{eqnarray}
Thanks to \eqref{vinf}, we have
\begin{equation}\label{sob}||u(t)||_\infty\leq C_1\ ||u(t)||_{W^{1,p}}\leq R\ C_1\leq C_2 .\end{equation}
Therefore, using elliptic regularity (see Theorem \ref{reg}) and \eqref{sob}, we obtain 
\begin{equation}\label{elli}||\bm{\omega}_u(t)||_{ W^{2,p}}\leq \frac{K(p)}{\varepsilon}\ ||\nabla E(u(t))||_p\leq C\ ||E'||_{L^\infty(-C_2, C_2)}\ R\leq C.\end{equation}
Using again \eqref{sob} along with \eqref{elli} we find
\begin{eqnarray}
||\Lambda u(t)||_p&\leq & C_1\ ||u_0||_p+C\ \int^t_0(t-s)^{-\frac{1}{2}} ||u(s)||_\infty\ ||\bm{\omega}_u(s)||_p\ ds\nonumber\\
&+&r\ \int^t_0 ||u(s)||_p\ ||E(u(s))||_{\infty} \ ds\nonumber\\
&\leq& C_1\ ||u_0||_p+C\ \int_0^t(t-s)^{-\frac{1}{2}}\ ds+r\ T \ R\ ||E||_{L^\infty(-C_2, C_2)}\nonumber\\
&\leq& C_1\ ||u_0||_p+C \ t^{\frac{1}{2}}+C\ T\nonumber\\
&\leq& C_1\ ||u_0||_p+C_3\ T^{\frac{1}{2}}\label{n1}
\end{eqnarray}
(recall that $T\leq 1$).
On another hand, by \eqref{semigr} we have
\begin{eqnarray}
||\nabla \Lambda u(t)||_p&\leq& C_1\ ||\nabla u_0||_p+\int^t_0||\nabla e^{(t-s)\ (\delta\ \Delta)}\ \nabla\cdot(u\ \bm{\omega}_u)(s)||_p\ ds\nonumber\\
&+&r\ \int^t_0 ||e^{(t-s)(\delta\ \Delta)}\ \nabla(u\ E(u))(s)||_p\ ds\nonumber\\
&\leq&C_1\ ||\nabla u_0||_p+\delta^{-\frac{1}{2}}\int^t_0 (t-s)^{-\frac{1}{2}}\ ||(\nabla u\cdot \bm{\omega}_u+u\ \nabla\cdot \bm{\omega}_u)(s)||_p\ ds\nonumber\\
&+&r\ \int^t_0 ||\nabla ( \tilde{E}(u))(s)||_p\ ds.\nonumber
\end{eqnarray}
Since $u\in X_R(T)$, using \eqref{sob} we can see that
\begin{equation*}r\ ||\nabla (\tilde{E}(u))||_p\leq r\ ||\tilde {E}'(u)\ \nabla u||_p\leq r\ ||\tilde{E}'||_{L^\infty(-C_2, C_2)}\ ||\nabla u||_p\leq C_4,\end{equation*}
which gives that
\begin{eqnarray}
||\nabla \Lambda u(t)||_p&\leq& C_1\ ||\nabla u_0||_p+\delta^{-\frac{1}{2}}\int_0^t (t-s)^{-\frac{1}{2}} ||\bm{\omega}_u||_\infty\ ||\nabla u||_p\ ds\nonumber\\
&+&\delta^{-\frac{1}{2}}\int^t_0(t-s)^{-\frac{1}{2}}(||\nabla\cdot \bm{\omega}_u||_p\ ||u||_\infty)\ ds
+ C_4\ T.\nonumber
\end{eqnarray}
Since 
\begin{equation}
||\bm{\omega}_u||_\infty \leq C_1\ ||\bm{\omega}_u||_{W^{1, p}}\leq C_5,\label{n3}
\end{equation}
by \eqref{elli} and \eqref{vinf}, we use once more \eqref{sob} and obtain that
\begin{eqnarray}
||\nabla \Lambda u(t)||_p&\leq& C_1\ ||\nabla u_0||_p+C\ \int^t_0(t-s)^{-\frac{1}{2}}\ ds
+C\ \int^t_0(t-s)^{-\frac{1}{2}}\ ds+C_4\ T\nonumber\\
&\leq& C_1\ ||\nabla u_0||_p+C_6\ T^{\frac{1}{2}}.\label{n2}
\end{eqnarray}
Combining \eqref{n1} and \eqref{n2} we get
$$\sup_{t\in[0,T]}||\Lambda u(t)||_{W^{1,p}}\leq C_1\ ||u_0||_{W^{1,p}}+C_7\ T^{\frac{1}{2}}.$$
Choosing $R=2\ C_1\ ||u_0||_{W^{1,p}}$ and $T\in (0,1)$ such that 
$$ C_1\ ||u_0||_{W^{1,p}}+C_7\ T^{\frac{1}{2}}\leq R,$$ we obtain that
$$\sup_{t\in[0,T]}||\Lambda u(t)||_{W^{1,p}}\leq R.$$
It follows that $\Lambda$ maps $X_R(T)$ into itself.
\item Step 2. We next show that $\Lambda$ is a strict contraction for $T$ small enough.\\

Let $u$ and $v$ be two functions in $X_R(T)$. Using \eqref{semigr} we have
\begin{eqnarray}
||\Lambda u(t)-\Lambda v(t)||_p &\leq& \int^t_0 \left\lvert\left\lvert\nabla e^{(t-s)\ (\delta\ \Delta)}\ [-u\ \bm{\omega}_u+v\ \bm{\omega}_v]\right\rvert\right\rvert_p\ ds\nonumber\\
&+&r\ \int_0^t \left\lvert\left\lvert e^{(t-s)(\delta\ \Delta)}\ [\tilde{E}(u)-\tilde{ E}(v)]\right\rvert\right\rvert_p\ ds\nonumber\\
&\leq& C_1\ \delta^{-\frac{1}{2}}\int_0^t (t-s)^{-\frac{1}{2}}\ \left\lvert\left\lvert -u\ \bm{\omega}_u+v\ \bm{\omega}_v\right\rvert\right\rvert_p\ ds\nonumber\\
&+&r\ C_1\int_0^t \left\lvert\left\lvert \tilde{E}(u)-\tilde{E}(v)\right\rvert\right\rvert_p\ ds.\label{se}
\end{eqnarray}
Note that, by \eqref{sob} and \eqref{n3}, we have
\begin{eqnarray}
||u\ \bm{\omega}_u-v\ \bm{\omega}_v||_p&\leq& ||u\ \bm{\omega}_u-u\ \bm{\omega}_v-v\ \bm{\omega}_v+u\ \bm{\omega}_v||_p\nonumber\\
&\leq& ||u||_\infty\ ||\bm{\omega}_u-\bm{\omega}_v||_p+||\bm{\omega}_v||_\infty\ ||u-v||_p\nonumber\\
&\leq& C\ ||\bm{\omega}_u-\bm{\omega}_v||_p+C\ ||u-v||_p ,\label{si}
\end{eqnarray}
and it follows from Theorem \ref{reg} and \eqref{sob} that
\begin{eqnarray}
 ||\bm{\omega}_u-\bm{\omega}_v||_{W^{2,p}}&\leq& C\ ||\nabla E(u)-\nabla E(v)||_p\nonumber\\
&\leq& C\ ||E'(u)\ \nabla u-E'(u)\ \nabla v-E'(v)\ \nabla v+E'(u)\ \nabla v||_p\nonumber\\
&\leq& C\ ||E'||_{L^\infty(-C_2, C_2)}\ ||\nabla u-\nabla v||_p\nonumber\\
&+&C||E'(v)-E'(u)||_{\infty}\ ||\nabla v||_p\nonumber\\
&\leq& C\ ||\nabla v-\nabla u||_p+C\ ||E''||_{L^\infty(-C_2, C_2)}\ d_X(u,v).\nonumber\\
&\leq&C_8\ d_X(u,v).\label{ce}
\end{eqnarray}
Combining \eqref{ce} and \eqref{si} we obtain
\begin{eqnarray}
&&\int_0^t(t-s)^{-\frac{1}{2}}\ ||-u\ \bm{\omega}_u+v\ \bm{\omega}_v||_p\ ds\nonumber\\
&\leq& C\ T^{\frac{1}{2}}\ C_8\ d_X(u,v)+T^{\frac{1}{2}}\ C_9\ d_X(u,v)\nonumber\\
&\leq& T^{\frac{1}{2}}\ C_{10}\ d_X(u,v)\label{n4}.
\end{eqnarray}
 Since $u$ and $v$ are bounded by \eqref{sob}, we have
\begin{eqnarray}
r\ ||\tilde{E}(u)-\tilde{E}(v)||_p&\leq& C\ ||u-v||_p.\nonumber
\end{eqnarray}
Then, we get
\begin{eqnarray}
\int_0^t r\ ||\tilde{E}(u)-\tilde{E}(v)||_p(s)\ ds &\leq& C_{11}\ T\ d_X(u,v)\nonumber
\end{eqnarray}
Substituting \eqref{n4} and the above inequality in \eqref{se} we conclude that
\begin{equation*}
||\Lambda u(t)-\Lambda v(t)||_p\leq C_{12}\ T^{\frac{1}{2}}\  d_X(u,v).
\end{equation*}
Using again \eqref{semigr}, we have
\begin{eqnarray}
||\nabla \Lambda u(t)-\nabla \Lambda v(t)||_p&\leq& \int_0^t \left\lvert\left\lvert \nabla e^{(t-s)\ (\delta\ \Delta)}\ \left[-\nabla\cdot(u\ \bm{\omega}_u)+\nabla \cdot(v\ \bm{\omega}_v)\right]\right\rvert\right\rvert_p\ ds\nonumber\\
&+& r\int_0^t\left\lvert\left\lvert \nabla e^{(t-s)\ (\delta\ \Delta)}\ (\tilde{E}(u)-\tilde{E}(v))\right\rvert\right\rvert_p\ ds\nonumber\\
&\leq& \delta^{-\frac{1}{2}}\ C_1\ \int_0^t (t-s)^{-\frac{1}{2}}\ \left\lvert\left\lvert -\nabla\cdot(u\ \bm{\omega}_u)+\nabla\cdot(v\ \bm{\omega}_v)\right\rvert\right\rvert_p\ ds\nonumber\\
&+& r\ C_1\ \int_0^t \left\lvert\left\lvert \nabla \left(\tilde{E}(u)-\tilde{E}(v)\right)\right\rvert\right\rvert_p(s)\ ds.\label{der}
\end{eqnarray}

Since the mapping
\begin{center}
$
\begin{array}{ccc}
\displaystyle W^{1,p}(\Omega)\times W^{1,p}(\Omega)&\longrightarrow & W^{1,p}(\Omega)\\
\displaystyle u, v&\longmapsto & u\ v
\end{array}
$
\end{center}
is bilinear and continuous due to $p>N$, we deduce from \eqref{elli} and \eqref{ce} that
\begin{eqnarray}
\delta^{-\frac{1}{2}}\ C_1\ ||-\nabla\cdot(u\ \bm{\omega}_u)+\nabla\cdot(v\ \bm{\omega}_v)||_p&\leq& ||u\ \bm{\omega}_u-v\ \bm{\omega}_v||_{W^{1,p}}\nonumber\\
&\leq& C\ ||u||_{W^{1,p}}\ ||\bm{\omega}_u-\bm{\omega}_v||_{W^{1,p}}\nonumber\\
&+&C\ ||\bm{\omega}_v||_{W^{1,p}}\ ||u-v||_{W^{1,p}}\nonumber\\
&\leq& C\ d_X(u,v)+ C\ ||u-v||_{W^{1,p}}\nonumber\\
&\leq& C_{13}\ d_X(u,v)\nonumber
\end{eqnarray}
Thus, 
$$\delta^{-\frac{1}{2}}\ C_1\ \int_0^t (t-s)^{-\frac{1}{2}}\ ||-\nabla\cdot(u\ \bm{\omega}_u)+\nabla\cdot(v\ \bm{\omega}_v)||_p\ ds\leq T^{\frac{1}{2}}\ C_{13}\ d_X(u,v).$$
On the other hand, due to \eqref{sob} and the embedding of $W^{1,p}(\Omega)$ in $L^\infty(\Omega)$
\begin{eqnarray}
||\nabla (\tilde{E}(u)-\tilde{E}(v))||_p&=& ||\tilde{E}'(u)\ \nabla u-\tilde{E}'(v)\ \nabla v||_p\nonumber\\
&\leq& C\ ||\tilde{E}'(u)||_{\infty}\ ||\nabla u-\nabla v||_p+C\ ||\tilde{E}'(u)-\tilde{E}'(v)||_\infty\ ||\nabla v|| _p,\nonumber\\
&\leq& C\ ||\tilde{E}'||_{L^\infty(-C_2,C_2)}\ ||\nabla u-\nabla v||_p+C\ ||\tilde{E}''||_{L^\infty(-C_2,C_2)}\ ||u-v||_\infty\nonumber\\
&\leq& C \ ||u-v||_{W^{1,p}}\nonumber.
\end{eqnarray}
Then
$$r\ C_1\ ||\nabla (\tilde{E}(u)-\tilde{E}(v))||_p\leq C_{14}\ ||u-v||_{W^{1,p}}.$$

Therefore,
$$||\nabla \Lambda u(t)-\nabla \Lambda v(t)||_p\leq T^{\frac{1}{2}}\ C_{13}\ d_X(u,v)+ T\ C_{14}\ d_X(u,v)\leq T^{\frac{1}{2}}\ C_{15}\ d_X(u,v). $$

Finally we get
$$d_X(\Lambda u(t)-\Lambda v(t))\leq T^{\frac{1}{2}}\ C_{16}\ d_X(u,v).$$

Choosing $T\in (0,1)$ such that $T^{\frac{1}{2}}\ C_{16}<1$ we obtain that $\Lambda$ is indeed a strict contraction in $X_R(T)$ and thus has a unique fixed point $u$.\\

Furthermore, since $u\in C\left( [0,T], W^{1,p}(\Omega)\right)$ and $p>N$, we have $\nabla E(u)\in C\left( [0,T], L^p(\Omega)\right)$ and we infer from Theorem \ref{reg} that $\bm{\omega}_u \in C\left( [0,T], W^{2,p}(\Omega)\right)$. Combining this property with the fact that $W^{1,p}(\Omega)$ is an algebra, we realize that both $\nabla\cdot(u\ \bm{\omega}_u)$ and $u\ E(u)$ belong to $C\left( [0,T], L^p(\Omega)\right)$. Classical regularity properties of the heat equation then guarantee that $u\in C\left( (0,T], W^{2,p}(\Omega)\right)$ and is a strong solution to \eqref{eqdeu}.
\item Step 3. Thanks to the analysis performed in Steps $1$ and $2$, the existence and uniqueness of a maximal solution follows by classical argument, see \cite{anintroduction} for instance.
\item Step 4. Since $0$ clearly solves \eqref{eqdeu}, and $u_0\geq 0$, the positivity of $u$ follows from the comparison principle.
\end{itemize}
\end{proof}
\section{Global existence}
From now on we choose $N=1$, $\Omega=(-1,1)$, $p=2$ and  we set $\varphi=\bm{\omega}_u$ to simplify the notation. 
\subsection{The bistable case: $E(u)=(1-u)(u-a)$.}
In this case, the system \eqref{grin} now reads 
\begin{equation}
\label{bc}
\left\{
\begin{array}{llll}
\displaystyle \partial_t u&=& \delta\ \partial^2_{xx} u-\partial_x(u\ \varphi)+r\ u\ (u-a)(1-u),& x\in (-1,1), t>0 \\
\displaystyle -\varepsilon\ \partial^2_{xx} \varphi+\varphi&=&(-2u+(a+1))\ \partial_x u,& x\in (-1,1), t>0 \\
\displaystyle \partial_x u(t,\pm 1)&=&\varphi(t,\pm1)=0,& t>0,\\
\displaystyle u(0,x)&=&u_0(x),& x\in (-1,1),
\end{array}
\right.
\end{equation}
for a some $a\in (0,1)$.\\

Since $E\in C^2(\mathbb{R})$, Theorem \ref{local} ensures that there is a maximal solution\\ $u$ of \eqref{bc} in $C\left( [0, T_{\mathrm{max}}), W^{1,p}((-1,1))\right)\cap C\left( (0, T_{\mathrm{max}}), W^{2,p}(-1,1)\right)$ .\\

To prove Theorem \ref{the1} we show that, for all $T>0$ and $t\in [0, T]\cap [0, T_{\mathrm{max}})$, $u(t)$ is bounded in $W^{1, 2}(-1,1)$.\\
We begin the proof by the following lemmas which give some estimates on $u$ and $\varphi$.
\begin{lemma}\label {estu}Let the same assumptions as that of Theorem \ref{the1} hold, and $u$ be the nonnegative maximal solution of \eqref{bc}. Then for all $T>0$ there exists $C_1(T)$, such that $u$ and $\varphi$ satisfy the following estimates
\begin{equation}\label{u}
||u(t)||_2\leq C_1(T),\ \ \mathrm{for\ all}\ t\in [0,T]\cap[0, T_{\mathrm{max}}),
\end{equation}
\begin{equation}\label{dxu}
\int_0^t ||\partial_x u||_2^2\ dt\leq C_1(T)\ \ \mathrm{for\ all}\ t\in [0,T]\cap[0, T_{\mathrm{max}}),
\end{equation}
and
\begin{equation}\label{fi}
\int_0^t ||\varphi(t)||^2_{W^{1,2}}\ dt\leq C_1(T)\ \ \mathrm{for\ all}\ t\in [0,T]\cap[0, T_{\mathrm{max}}).
\end{equation}
\end{lemma}
\begin{proof}
Multiplying the first equation in \eqref{bc} by $u(t)$ and integrating it over $(-1,1)$, we obtain 
\begin{equation}\label{cub}
\frac{\mathrm{d}}{\mathrm{d}t}\int_{-1}^1|u|^2\ dx=-2\ \delta\int_{-1}^1|\partial_x u|^2\ dx+2\ \int_{-1}^1\ u\ \varphi\ \partial_x u\ dx+2\ r\ \int_{-1}^1 u^2\ E(u)\ dx.
\end{equation}
Multiplying now the second equation in \eqref{bc} by $\varphi$ and integrating it over $(-1,1)$ we obtain 
\begin{equation}\label{phi}
\varepsilon\ \int_{-1}^1|\partial_x\  \varphi|^2\ dx+\int_{-1}^1|\varphi|^2\ dx=-2\ \int_{-1}^1 u\ \varphi\ \partial_x u\ dx+(a+1)\int_{-1}^1\partial_x u\ \varphi\ dx.
\end{equation}
At this point we notice that the cubic terms on the right hand side of \eqref{cub} and \eqref{phi} cancel one with the other, and summing \eqref{phi} and \eqref{cub} we obtain 
\begin{equation}
\frac{\mathrm{d}}{\mathrm{d}t}||u||^2_2+\varepsilon\ ||\partial_x\varphi||_2^2+||\varphi||^2_2+2\ \delta\ ||\partial_x u||_2^2=2\ r \int_{-1}^1 u^2\ E(u)\ dx+ (a+1)\int_{-1}^1 \partial_x u\ \varphi\ dx.
\end{equation}
We integrate by parts and use Cauchy-Schwarz inequality to obtain
\begin{equation}
(a+1)\ \int_{-1}^1 \partial_x u\ \varphi\ dx=-(a+1)\int_{-1}^1 u\ \partial_x \varphi\ dx\nonumber\\
\leq\frac{(a+1)^2}{2\ \varepsilon}\ ||u||_2^2+\frac{\varepsilon}{2}\ ||\partial_x \varphi||_2^2.\nonumber
\end{equation}
On the other hand, $u^2\ E(u)\leq 0$ if $u\notin (a,1)$ so that
$$\int_{-1}^1 u^2\ E(u)\ dx\leq 2\ (1-a)$$
The previous inequalities give that
\begin{equation*}\label{df}\frac{\mathrm{d}}{\mathrm{d}t}||u||^2_2+\frac{\varepsilon}{2}\ ||\partial_x\varphi||_2^2+||\varphi||^2_2+2\ \delta\ ||\partial_x u||_2^2\leq\frac{(a+1)^2}{2\ \varepsilon}\ ||u||_2^2+ 4\ r\ (1-a).\end{equation*}
Therefore, for all $T>0$ there exists $C_1(T)$ such that \eqref{u}, \eqref{dxu} and \eqref{fi} hold.
\end{proof}
\begin{lemma}\label{infiniu}
Let the same assumptions as that of Theorem \ref{the1} hold, and $u$ be the nonnegative maximal strong solution of \eqref{bc}. For all $T>0$, there is $C_\infty(T)$ such that 
\begin{equation}\label{infu}
||u(t)||_\infty\leq C_\infty(T),\ \ \mathrm{for\ all}\ t\in[0, T]\cap [0, T_{\mathrm{max}}).
\end{equation}
\end{lemma}
\begin{proof}
The estimates \eqref{u} and \eqref{dxu} and the Gagliardo-Nirenberg inequality \eqref{gagliardo} yield that there exists $C_2(T)$ such that
$$\int_0^t ||u\ \partial_x u||_2\ ds\leq C\ \int_0^t ||u||_2^{\frac{1}{2}}\ ||\partial_x u||_2^{\frac{3}{2}}\ ds\leq C_2(T)\ \ \mathrm{for\ all}\ t\in [0,T]\cap[0, T_{\mathrm{max}}).$$
The second equation in \eqref{bc} and classical elliptic regularity theory ensure that  there exists $C(T)$ such that
$$\int_0^t||\varphi||_{W^{2,2}}\ ds\leq C(T)\ \ \mathrm{for\ all}\ t\in [0,T]\cap[0, T_{\mathrm{max}}), $$
which gives in particular, since $W^{2,2}(-1,1)$ is embedded in $W^{1, \infty}(-1,1)$,
\begin{equation}\label{dxfi}\int_0^t||\partial_x\varphi||_\infty\ ds\leq C_3(T)\ \ \mathrm{for\ all}\ t\in [0,T]\cap[0, T_{\mathrm{max}}).\end{equation}
Now, we multiply the first equation in \eqref{bc} by $q\ u^{q-1}$ where $q>1$ and integrate it over $(-1,1)$ to obtain 
\begin{eqnarray}
\frac{\mathrm{d}}{\mathrm{d}t}\int_{-1}^1 |u|^q\ dx&=&-4 \frac{\delta\ (q-1)}{q}\int_{-1}^1|\partial_x u^{\frac{q}{2}}|^2\ dx+(q-1)\ \int_{-1}^1 \varphi\ \partial_x u^q\ dx\nonumber\\
&+&r\ q\ \int_{-1}^1\ u^q\ (u-a)\ (1-u)\ dx\nonumber\\
&\leq& - (q-1)\ \int_{-1}^1 \partial_x \varphi \ u^q\ dx+r\ q\ 2\ (1-a)\nonumber
\end{eqnarray} 
Using H\"older's inequality , we obtain
\begin{equation}\label{n5}
\frac{\mathrm{d}}{\mathrm{d}t}||u||_q^q\leq  (q-1)\ ||\partial_x\varphi||_\infty\ ||u||_q^q+2\ q\ r.
\end{equation}
Introducing
$$\phi(t)=\int_0^t ||\partial_x\varphi(s)||_\infty\ ds\leq C_3(T) \ \mathrm{for\ all}\ t\in [0,T]\cap[0, T_{\mathrm{max}}), $$ the bound being a sequence of \eqref{dxfi}, we integrate \eqref{n5} and find
\begin{eqnarray*}
||u(t)||_q^q&\leq&||u_0||_q^q\ e^{(q-1)\ \phi(t)}+2\ q\ r\ \int_0^t e^{(q-1)\ (-\phi(s)+\phi(t))}\ ds\\
&\leq&(||u_0||_q^q+2\ q\ r)\ T\ e^{q\ C(T)},\\
||u(t)||_q&\leq&\left((||u_0||_q^q+2\ q\ r)\ T\right)^{\frac{1}{q}}\ e^{C(T)}\ \ \mathrm{for\ all}\ t\in [0,T]\cap[0, T_{\mathrm{max}}).
\end{eqnarray*}
Consequently, by letting $q$ tend to $\infty$, we see that there exists $C_\infty(T)$ such that
\begin{equation*}\label{inf}||u(t)||_\infty\leq C_\infty(T),\ \mathrm{for\ all}\ t\in [0,T]\cap[0, T_{\mathrm{max}}).\end{equation*}
\end{proof}
\begin{lemma}\label{estdif}
Let the same assumptions as that of Theorem \ref{the1} hold, and $u$ be the nonnegative maximal strong solution of \eqref{bc}. For all $T>0$, there is $C_4(T)$ such that 
\begin{equation}
||\partial_x u(t)||_2\leq C_4(T)\ \ \mathrm{for\ all}\ t\in [0,T]\cap[0, T_{\mathrm{max}}).
\end{equation}
\end{lemma}
\begin{proof}
We multiply the first equation in \eqref{bc} by $(-\partial^2_{xx} u)$ and integrate it over $(-1,1)$ to obtain
\begin{eqnarray}
\frac{1}{2}\ \frac{\mathrm{d}}{\mathrm{d}t}\int_{-1}^1 |\partial_x u|^2\ dx&=& -\delta\ \int_{-1}^1 |\partial^2_{xx} u|^2\ dx+\int_{-1}^1 \partial_x(u\ \varphi)\ \partial^2_{xx}u\ dx\nonumber\\
&+&r\int_{-1}^1 u\ (1-u)\ (u-a)\ (-\partial_{xx}^2 u)\ dx\nonumber\\
&=&  -\delta\ \int_{-1}^1 |\partial^2_{xx} u|^2\ dx+\int_{-1}^1 (u\ \partial_x\varphi+\partial_x u\ \varphi)\ \partial^2_{xx}u\ dx\nonumber\\
&+&r\int_{-1}^1 u\ (-u^3+(a+1)u^2-a)\ \partial^2_{xx}u\ dx.\nonumber
\end{eqnarray}
Using Cauchy-Schwarz inequality and Lemma \ref{infiniu} we obtain,
\begin{eqnarray}
\frac{1}{2}\ \frac{\mathrm{d}}{\mathrm{d}t}\int_{-1}^1 |\partial_x u|^2\ dx&\leq& -\delta\  \int_{-1}^1 |\partial^2_{xx} u|^2\ dx+\frac{\delta}{3}\int_{-1}^1 |\partial^2_{xx} u|^2\ dx+C\ {||u||_\infty^2}\ \int_{-1}^1 |\partial_x \varphi|^2\ dx\nonumber\\
&+&\frac{\delta}{3}\int_{-1}^1 |\partial^2_{xx} u|^2\ dx+C\ {||\varphi||^2_\infty}\ \int_{-1}^1 |\partial_x u|^2\ dx\nonumber\\
&+&\frac{\delta}{3}\int_{-1}^1 |\partial^2_{xx} u|^2\ dx+C\ ||(u\ (-u^3+(a+1)u^2-a))||_\infty^2 \nonumber\\
&\leq& C(T)\ \int_{-1}^1 |\partial_x \varphi|^2\ dx+C\ {||\varphi(t)||^2_\infty}\ \int_{-1}^1 |\partial_x u|^2\ dx+C(T).\label{dif}
\end{eqnarray}
Using \eqref{fi} and Sobolev embedding theorem we obtain the following estimate 
\begin{equation}\label{fff}\int_0^t||\varphi||_\infty^2\ ds\leq C(T) \ \ \mathrm{for\ all}\ t\in [0,T]\cap[0, T_{\mathrm{max}}).\end{equation}
 Since \eqref{fi} and \eqref{fff} hold, then it follows from \eqref{dif} after integration that
\begin{equation*}\label{limin}||\partial_x u(t)||_2^2\leq C_4(T),\ \mathrm{for\ all}\ t\in [0,T]\cap[0, T_{\mathrm{max}}).\end{equation*}
\end{proof}
It remains to prove Theorem \ref{the1}.
\begin{proof}[Proof of Theorem \ref{the1}] 
For all $T>0$, Lemma \ref{estdif} and the estimate \eqref{u} ensure that
$$||u(t)||_{W^{1,2}}\leq C(T),\ \ \mathrm{for\ all}\ t\in [0,T]\cap[0, T_{\mathrm{max}}),$$
 which guarantees that $u$ cannot explode in $W^{1,2}(-1,1)$ in finite time and thus\\
  that $T_{\mathrm{max}}=\infty$.
\end{proof}
\subsection{The monostable case: $E(u)=1-u$}
For this choice of $E$, system \eqref{grin} now reads
\begin{equation}
\label{grin2}
\left\{
\begin{array}{llll}
\displaystyle \partial_t u&=&\delta\ \partial^2_{xx} u-\partial_x(u\ \varphi)+r\ u\ (1-u)& x\in (-1,1),\ t>0 \\
\displaystyle -\varepsilon\ \partial^2_{xx} \varphi+\varphi&=&-\ \partial_x u,& x\in (-1,1),\ t>0 \\
\displaystyle \partial_xu(t,\pm1)=\varphi(t,\pm1)&=&0& t>0,\\
\displaystyle u(0,x)&=&u_0(x)& x\in (-1,1),
\end{array}
\right.
\end{equation}
\\
Since $E\in C^2(\mathbb{R}) $, Theorem \ref{local} ensures that there is a maximal solution $u$ of \eqref{grin2} in $ C\left( [0, T_{\mathrm{max}}), W^{1,p}(-1,1)\right)\cap C\left( (0, T_{\mathrm{max}}), W^{2,p}(-1,1)\right) $.\\

In contrast to the previous case, it does not seem to be possible to begin the global existence proof with an $L^\infty(L^2)$ estimate on $u$. Nevertheless, there is still a cancellation between the two equations which actually gives us an $L^\infty(L \log L)$ bound on $u$ and a $L^2$ bound on $\partial_x\sqrt u$. Integrating \eqref{grin2} over $(0,t)\times(-1,1)$ and using the nonnegativity of $u$, we first observe that,
\begin{equation}\label{norm11}
||u(t)||_1\leq ||u_0||_1+2\ r\ t,\ \ \mathrm{for\ all}\ t\in [0, T_{\mathrm{max}}).
\end{equation}

To prove Theorem \ref{theo2} we need to prove the following lemmas:
\begin{lemma}\label{estimate}
Let the same assumptions as that of Theorem \ref{theo2} hold, and let  $u$ be the maximal solution of \eqref{grin2}. Then for all $T>0$, there exists a constant $C_1(T)$ such that
\begin{equation}\label{dxru}
\int_0^t ||\partial_x \sqrt{u}||_2^2\ ds\leq C_1(T),\ \ \mathrm{for\ all}\ t\in [0,T]\cap[0, T_{\mathrm{max}}),
\end{equation}
and
\begin{equation}\label{fiw}
\int_0^t ||\varphi||^2_{W^{1,2}}\ ds\leq C_1(T),\ \ \mathrm{for\ all}\ t\in [0,T]\cap[0, T_{\mathrm{max}}).
\end{equation}
\end{lemma}
\begin{proof}
The proof goes as follows. On the one hand, we multiply the first equation in \eqref{grin2} by $(\log u+1)$ and integrate it over $(-1,1)$. Since $u\ (1-u)\ \log u\leq 0$ and $u\ (1-u)\leq 1$,
\begin{eqnarray}
\frac{\mathrm{d}}{\mathrm{d}t}\int_{-1}^1 u\ \log u\ dx&=& -\int_{-1}^1 (\delta\ \partial_x u-u\ \varphi)\ (\frac{1}{u}\ \partial_x u)\ dx+r\ \int_{-1}^1 u\ (1-u)\ (\log u+1)\ dx\nonumber\\
&\leq& -\int_{-1}^1 \frac{\delta}{u}\ (\partial_x u)^2\ dx+\int_{-1}^1\varphi\ \partial_xu\ dx+2\ r.\label{uu}
\end{eqnarray}
On the other hand, we multiply the second equation in \eqref{grin2} by $\varphi$ and integrate it over $(-1,1)$ to obtain
\begin{equation}\label{fi2}\varepsilon \int_{-1}^1 |\partial_x \varphi|^2\ dx+\int_{-1}^1 |\varphi|^2\ dx=-\int_{-1}^1 \partial_x u\ \varphi\ dx.\end{equation}
Adding \eqref{fi2} and \eqref{uu} yields
\begin{equation}\label{hb}
\frac{\mathrm{d}}{\mathrm{d}t}\int_{-1}^1 u\ \log u\ dx+\varepsilon\ ||\partial_x\varphi||_2^2+||\varphi||_2^2\leq -4\ \delta\int_{-1}^1 |\partial_x \sqrt{u}|^2\ dx+2\ r.
\end{equation}
Finally, \eqref{dxru} and \eqref{fiw} are obtained by a time integration of \eqref{hb}.
\end{proof}
\begin{lemma}\label{lolo}
Let the same assumptions as that of Theorem \ref{theo2} hold, and let  $u$ be the maximal solution of \eqref{grin2}. Then for all $T>0$, there exists a constant $C_2(T)$ such that
\begin{equation}\label{u2}||u(t)||_2\leq C_2(T),\ \ \mathrm{for\ all}\ t\in [0,T]\cap[0, T_{\mathrm{max}}),\end{equation}
and
\begin{equation}\label{dxu2}
\int_0^t\ ||\partial_x u||_2^2\  ds\leq C_2(T),\ \ \mathrm{for\ all}\ t\in [0,T]\cap[0, T_{\mathrm{max}}).
\end{equation}
\end{lemma}
\begin{proof}
A simple computation shows that, since $u^2\ (1-u)\leq 1$,
 \begin{eqnarray}
 \frac{\mathrm{d}}{\mathrm{d}t} \int_{-1}^1 u^2\ dx&\leq&-2\ \delta\ \int_{-1}^1 |\partial_x u|^2\ dx+2\  \int_{-1}^1 u\ \varphi\ \partial_x u\ dx+4\ r.\label{za}
 \end{eqnarray}
Using Cauchy-Schwarz inequality, Gagliardo-Nirenberg inequality \eqref{gagliardo}, Young inequality  and \eqref{norm11} we obtain that for all $T>0$,
 \begin{eqnarray*}
 2\int_{-1}^1 u\ \varphi\ \partial_x u\ dx&=&-\int_{-1}^1 u^2\ \partial_x\varphi\ dx
 \leq ||u||_4^2\ ||\partial_x\varphi||_2\nonumber\\
 &\leq &C\ \left(||u||_{W^{1,2}}^{\frac{1}{2}}\ ||u||_1^{\frac{1}{2}}\right)^2\ ||\partial_x\varphi||_2
 \leq C(T)\ ||\partial_x\varphi||_2\ ||u||_{W^{1,2}}\\
 &\leq& \delta\ ||u||_{W^{1,2}}^2+\ C(T)\ ||\partial_x\varphi||_2^2
 \end{eqnarray*}
 We substitute the previous inequality in \eqref{za} to obtain 
 \begin{equation}\label{ccc}
\frac{\mathrm{d}}{\mathrm{d}t} ||u||_2^2\leq-2\ \delta\  ||\partial_x u||_2^2+\delta\ ||u||_{W^{1,2}}^2+\ C(T)\ ||\partial_x\varphi||_2^2+4\ r.
\end{equation}
Integrating \eqref{ccc} in time, and using \eqref{fiw} yield that there exists $C_3(T)$ such that \eqref{u2} and \eqref{dxu2} hold.
\end{proof}
 Now we are in a position to show the global existence of solution to \eqref{grin2}.
\begin{proof}[Proof of Theorem \ref{theo2} (global existence)]\

By elliptic regularity, and the continuous embedding of $W^{2,2}(-1,1)$ in $W^{1, \infty}(-1,1)$, we have
$$||\partial_x\varphi(t)||_\infty\leq C\ ||\varphi(t)||_{W^{2,2}}\leq C\ ||E'(u)\ \partial_xu||_2\leq C\ ||\partial_x u||_2,$$
which together with \eqref{dxu2}, implies that 
 $$\int_0^t||\partial_x \varphi(s)||^2_\infty\ ds\leq C\ \int_0^t||\partial_x u (s)||^2_2\ ds\leq  C(T),\ \ \mathrm{for\ all}\ t\in [0,T]\cap[0, T_{\mathrm{max}}).$$
 Thanks to this estimate, we now argue as in the proof of Lemma \ref{infiniu} and Lemma \ref{estdif} to get that
 $$||u(t)||_\infty+||\partial_xu(t)||_2\leq C(T), \ \ \mathrm{for\ all}\ t\in [0,T]\cap[0, T_{\mathrm{max}}).$$
Thus, the maximal solution $u$ of \eqref{grin2} cannot explode in finite time.
 \end{proof}
To complete the proof of Theorem \ref{theo2}, it remains to prove the asymptotic behaviour of $u$ when $t\rightarrow \infty$. We note that we have the following lemma which controls the $L^1(-1,1)$ norm of $u$. For $f\in L^1(-1,1)$, we set
 $$<f>=\frac{1}{2}\int_{-1}^1 f(x)\ dx$$
 \begin{lemma}
Let the same assumptions as that of Theorem \ref{theo2} hold, and let $u$ be the nonnegative global solution of \eqref{grin2}. For $r> 0$, there exists a constant $C_0>0$ such that 
 \begin{equation}
 \label{czero}
 0\leq<u(t)>\leq C_0,\ \ t\in (0,\infty),
 \end{equation}
 and if $r=0$
 \begin{equation}\label{norm1}
 <u(t)>=<u_0>=\frac{1}{2}||u_0||_1, \ \  t\in (0,\infty).
 \end{equation}
 \end{lemma}
 \begin{proof}
 We note that if $r=0$, $\frac{\mathrm{d}}{\mathrm{d}t}<u>=0$, so that \begin{equation*}<u(t)>=\frac{1}{2}\ ||u(t)||_1=\frac{1}{2}\ ||u_0||_1.\end{equation*}
 If $r>0$, 
 \begin{eqnarray}
 \frac{\mathrm{d}}{\mathrm{d}t}<u(t)>&=&r\ <u(t)>-r\ <u^2(t)>\nonumber\\
 &\leq& r\ <u(t)>-r\ <u(t)>^2,\nonumber
 \end{eqnarray}
 whence $<u(t)>\leq \max\left\{1, <u_0>\right\}$.
 \end{proof}
 Next we turn to the existence of a Liapunov functional for \eqref{grin2} which is the cornerstone of our analysis.
\begin{lemma}\label{lili}
Let the same assumptions as of that Theorem \ref{theo2} hold, and let $u$ be the nonnegative global solution of \eqref{grin2}. There exists a constant $C_1$ such that
\begin{equation}\label{dxsq}
 r\ \int_0^\infty\int_{-1}^1 u\ (u-1)\ \log u\ dx\ dt+\int_0^\infty \left(\ ||\partial_x \sqrt u||_2^2+\varepsilon\ ||\partial_x \varphi||_2^2+||\varphi||^2_2\ \right)\ dt\leq C_1,
 \end{equation}
 and
 \begin{equation}\label{infii}
 \int_0^\infty ||\varphi||^2_\infty\ ds\leq C_1.
 \end{equation}
\end{lemma}
\begin{proof}

 Let us define the following functional $L$
 $$L(u)=\int_{-1}^1\left( u\ \log u-u+1\right)\ dx\geq 0,$$
 and show that it is a Liapunov functional. Indeed
  \begin{equation}
 \frac{\mathrm{d}}{\mathrm{d}t} L(u)=-\delta\ \int_{-1}^1 \frac{|\partial_x u|^2}{u}\ dx+\int_{-1}^1 \varphi\ \partial_x u\ dx+r\int_{-1}^1 u\ \log u\ (1-u)\ dx.\label{liap}
 \end{equation}
 Combining \eqref{liap} and \eqref{fi2} we obtain that
 \begin{eqnarray}
 \frac{\mathrm{d}}{\mathrm{d}t} L(u)&=& - 4\ \delta\ \int_{-1}^1 |\partial_x\sqrt u|^2\ dx-\varepsilon \int_{-1}^1 |\partial_x \varphi|^2\ dx-\int_{-1}^1 |\varphi|^2\ dx+r\int_{-1}^1 u\ \log u\ (1-u)\ dx.\nonumber\\
 &=& -D(u,\varphi)\leq 0,\label{Du}
 \end{eqnarray}
 since $u\ \log u\ (1-u)\leq 0$. Then for all $t\geq 0$
 $$L(u(t))+\int_0^t D(u(s), \varphi(s))\ ds\leq L(u_0).$$
 Since $u_0$ and $u$ are nonnegative , we have
 \begin{eqnarray}
 L(u_0)&\leq& \int_{-1}^1\left(u_0^2+1\right)\ dx\leq ||u_0||_2^2+2,\nonumber
 \end{eqnarray}
 and
 \begin{equation*}
 L(u(t))\geq -\frac{2}{e}-\int_{-1}^1 u_0\ dx,
 \end{equation*}
 so that
\begin{equation}\label{born}
 \int_0^t D(u(s), \varphi(s))\ ds\leq 1+||u_0||_2^2+\frac{2}{e}+2\ <u_0>,\ \ t\geq0.
 \end{equation}
 Therefore, \eqref{born} yields there exists $C_1>0$ such that
 \begin{equation}\label{L1D}
 \int_0^\infty D(u, \varphi)\ dt\leq C_1.
 \end{equation}
 From \eqref{L1D}, we see that \eqref{dxsq} holds true. In addition, inequality \eqref{dxsq} together with Sobolev's embedding theorem give \eqref{infii}.
\end{proof}
 In the following lemma we show that $\left\{ u(t):\ t\geq0\right\}$ is bounded in $W^{1,2}(-1, 1)$.
 \begin{lemma}\label{lem}
 Let the same assumptions as that of Theorem \ref{theo2} hold, and let $u$ be the nonnegative global solution of \eqref{grin2}. Then $u$ belongs to $L^\infty\left((0, \infty); W^{1,2}(-1,1)\right)$.\end{lemma}
 \begin{proof}
First, 
\begin{eqnarray}
 \frac{\mathrm{d}}{\mathrm{d}t}\ ||u-<u>||_2^2&=&\frac{\mathrm{d}}{\mathrm{d}t}\ ||u||_2^2-\int_{-1}^1\partial_t u\ dx\ \int_{-1}^1 u\ dx\nonumber\\
 &=& \frac{\mathrm{d}}{\mathrm{d}t}\ ||u||_2^2-4\ r\ <u>^2+4\ r\ <u>\ <u^2>\nonumber\\
 &\leq&\frac{\mathrm{d}}{\mathrm{d}t}\ ||u||_2^2+2\ r\ C_0\ ||u||_2^2.\label{umoins}
 \end{eqnarray}
 Multiplying the first equation in \eqref{grin2} by $2\ u$, integrating it over $(-1,1)$, and using the Cauchy-Schwarz inequality and the fact that $u\geq 0$ we obtain 
 \begin{eqnarray}
 \frac{\mathrm{d}}{\mathrm{d}t}\ ||u||_2^2
 &\leq&-2 \ \delta \int_{-1}^1 |\partial_x u|^2\ dx-\int_{-1}^1 u^2\ \partial_x\varphi\  dx+2\ r \int_{-1}^1 u^2\ (1-u)\ dx\nonumber\\
 &\leq& -2 \ \delta\  ||\partial_x u||_2^2+||u||_4^2\ ||\partial_x\varphi||_2+2\ r\label{dut}.
 \end{eqnarray}
 Gagliardo-Nirenberg inequality \eqref{gagliardo} together with the Poincar\'e inequality \eqref{poincare} and \eqref{czero} give
 \begin{eqnarray}
 ||u||_4^2\leq C\ ||u||_{W^{1,2}}\  ||u||_1&\leq& C\ \left(\ ||\partial_xu||_2+||u||_1\ \right)\nonumber\\
 &\leq& C\left( ||\partial_xu||_2+1\right).\label{n6}
 \end{eqnarray}
Thus
 \begin{equation}\label{gag}
 ||u||_4^2\ ||\partial_x\varphi||_2\leq C\ \left( ||\partial_xu||_2+1\right)\ ||\partial_x\varphi||_2.
 \end{equation}
 
 Substituting  \eqref{gag}, \eqref{n6} and \eqref{dut} into \eqref{umoins}, and using Young and H\"older inequalities to obtain 
 \begin{eqnarray}
 \frac{\mathrm{d}}{\mathrm{d}t}\ ||u-<u>||_2^2&\leq& -2\ \delta\  ||\partial_x u||_2^2+C\ \left( ||\partial_xu||_2+1\right)\ ||\partial_x\varphi||_2+2r\nonumber\\
 &+&2r\ C_0\ ||u||_1^{\frac{2}{3}}\ ||u||_4^{\frac{4}{3}}\nonumber\\
 &\leq& -2\ \delta\  ||\partial_x u||_2^2+C\ ||\partial_x\varphi||_2^2+{\delta}\ ||\partial_xu||_2^2+C\nonumber.
 \end{eqnarray}
 Using Poincar\'e's inequality we get
\begin{eqnarray}
 \frac{\mathrm{d}}{\mathrm{d}t}\ ||u(t)-<u(t)>||_2^2+\alpha\ ||u(t)-<u(t)>||_2^2 &\leq& C\ ||\partial_x\varphi(t)||_2^2+C\nonumber,
 \end{eqnarray}
 for some $\alpha>0$ independent of $t$. Integrating this differential inequality gives
 \begin{eqnarray}
 e^{\alpha\ t}\ ||u(t)-<u(t)>||_2^2&\leq&||u_0-<u_0>||_2^2+\int_0^t e^{\alpha\ s}\left( C\ ||\partial_x\varphi(s)||_2^2+C\right)\ ds\nonumber\\
 ||u(t)-<u(t)>||_2^2&\leq& C\ e^{-\alpha\ t}+C\ \int_0^t e^{\alpha\ (s-t)}\ (||\partial_x\varphi(s)||_2^2+1)\ ds\nonumber 
 \end{eqnarray}
 Since $e^{-\alpha\ t}\leq 1$, and $e^{\alpha\ (s-t)}\leq 1$ as $s\leq t$ we obtain 
 \begin{eqnarray}
 ||u(t)-<u(t)>||_2^2&\leq& C+C\ \int_0^t||\partial_x\varphi(s)||_2^2\ ds+\frac{1}{\alpha},\nonumber 
 \end{eqnarray}
 Using  \eqref{dxsq} we end up with 
 \begin{equation}\label{hh}
 ||u(t)-<u(t)>||_2^2\leq C,
 \end{equation}
 where $C$ is independent of $t$. Therefore, $u$ belongs to $L^\infty((0,\infty); L^2(-1,1))$.\\
 
  It remains to show that $\partial_xu$ is in $L^\infty((0,\infty); L^2(-1,1))$. We multiply the first equation in \eqref{grin2} by $-\partial^2_{xx}u$ and integrate it over $(-1,1)$. Since $u\geq0$ we use Cauchy-Schwarz and Young inequalities and \eqref{hh} to obtain 
 \begin{eqnarray}
  \frac{1}{2}\frac{\mathrm{d}}{\mathrm{d}t}||\partial_x u||_2^2 &\leq& -\delta\ ||\partial^2_{xx} u||_2^2+\int_{-1}^1\ \partial^2_{xx} u\ (\partial_xu\ \varphi+\partial_x \varphi\ u)-r\int_{-1}^1 u\ (1-u)\ \partial^2_{xx} u\ dx\nonumber\\
  &\leq& -\delta\ ||\partial^2_{xx} u||_2^2-\frac{1}{2}\int_{-1}^1\ \partial_x\varphi\ |\partial_x u|^2\ dx+\frac{\delta}{8}\ ||\partial^2_{xx} u||_2^2\nonumber\\ &+&C\int_{-1}^1 |u|^2\ |\partial_x\varphi|^2\ dx -r\int_{-1}^1 u\ \partial_{xx}^2u\ dx-2\ r\ \int_{-1}^1 u\ |\partial_x u|^2\ dx\nonumber\\
  &\leq&-\delta\ ||\partial^2_{xx} u||_2^2+\frac{1}{2}\ ||\partial_x\varphi||_2\ ||\partial_xu||_4^2+\frac{\delta}{8}\ ||\partial^2_{xx} u||_2^2\nonumber\\
  &+&C\ ||u||_\infty^2\ ||\partial_x\varphi||_2^2+C+\frac{\delta}{8}\ ||\partial^2_{xx} u||_2^2.\label{dssu}
 \end{eqnarray}
 Since $\partial_x u\in W^{1,2}_0(-1,1)$, using the Gagliardo-Nirenberg inequality \eqref{gagliardo} and the classical Poincar\'e inequality \eqref{poincare}, we obtain
 \begin{equation}\label{care}
 ||\partial_xu||_4\leq C\ ||\partial_{xx}^2 u||_2^{\frac{1}{2}}\ ||\partial_xu||_1^{\frac{1}{2}}.
 \end{equation}
 Then, we substitute \eqref{care} into \eqref{dssu}, and by Young inequality, the Sobolev embedding, \eqref{czero} and \eqref{hh}, we obtain
 \begin{eqnarray}
 \frac{1}{2}\frac{\mathrm{d}}{\mathrm{d}t}||\partial_x u||_2^2 &\leq& -\frac{3\ \delta}{4}\ ||\partial^2_{xx} u||_2^2+{C}\ ||\partial_x\varphi||_2\ ||\partial_{xx}^2 u||_2 \ ||\partial_xu||_1+C\ ||u||_\infty^2\ ||\partial_x\varphi||_2^2+C\nonumber\\
 &\leq&-\frac{\delta}{2}\ ||\partial^2_{xx} u||_2^2+C\ ||\partial_x\varphi||_2^2\ ||\partial_xu||_1^2+C\ ||u||_{W^{1,2}}^2\ ||\partial_x\varphi||_2^2+C\nonumber\\
 &\leq&-\frac{\delta}{2}\ ||\partial^2_{xx} u||_2^2+C\ ||\partial_xu||_{2}^2\ ||\partial_x\varphi||_2^2+C\ (||\partial_x\varphi||^2_2+1)\label{di}.
 \end{eqnarray}
 Since $\partial_xu\in W^{1,2}_0(-1,1)$, we use once more the classical Poincar\'e inequality to obtain 
 $$\frac{\mathrm{d}}{\mathrm{d}t}||\partial_x u||_2^2+\beta||\partial_x u||_2^2\leq C\ ||\partial_xu||_{2}^2\ ||\partial_x\varphi||_2^2+C\ (||\partial_x\varphi||^2_2+1),$$
 for some $\beta>0$ independent of $t$.\\
 
 Define
 $$\phi(t)=\int_0^t ||\partial_x\varphi(s)||^2\ ds,\ \ t\geq0,$$
 and notice that , since $||\partial_x \varphi||_2^2$ belongs to $L^1(0, \infty)$ by \eqref{dxsq},
 $$0\leq\phi(t)\leq \phi_\infty=\int_0^\infty ||\partial_x \varphi(s)||_2^2\ ds. $$
 Integrating the previous differential inequality we find
 \begin{eqnarray}||\partial_x u(t)||_2^2&\leq& ||\partial_xu_0||_2^2\ e^{C\phi(t)-\beta\ t}\nonumber\\
 &+&C\int_0^t (1+\phi'(s))\ e^{\beta(s-t)+C \phi(t)-C\phi(s)}\ ds\nonumber\\
 &\leq&||\partial_x u_0||_2^2\ e^{C\ \phi_\infty}+C\ e^{C\ \phi_\infty}\ \int_0^t [e^{\beta(s-t)}+\phi'(s)]\ ds\nonumber\\
 &\leq& ||\partial_x u_0||_2^2\ e^{C\ \phi_\infty}+C\ e^{C\ \phi_\infty}\ \left(\frac{1}{\beta}+\phi_\infty\right).\nonumber
 \end{eqnarray}
Therefore, $\partial_xu$ belongs to $L^\infty\left((0, \infty), L^2(-1,1)\right)$, and Lemma \ref{lem} is proved.
 \end{proof}
 \begin{lemma}\label{lem2}
  Let the same assumptions as that of Theorem \ref{theo2} hold, and let $u$ be the nonnegative global solution of \eqref{grin2}. There is $C_2$ such that
  \begin{equation}\label{dtu}
  \int_0^\infty ||\partial_t u||_2^2\ dt\leq C_2.
  \end{equation}
 \end{lemma}
 \begin{proof}
 We multiply the first equation in \eqref{grin2} by $\partial_tu$ and integrate it over $(-1,1)$ to obtain
 \begin{eqnarray}
\int_{-1}^1 (\partial_t u)^2\ dx&=& \delta\ \int_{-1}^1 \partial_{xx}^2 u\ \partial_tu\ dx-\int_{-1}^1 \partial_x(u\ \varphi)\ \partial_tu\ dx+r\ \int_{-1}^1 u\ (1-u)\ \partial_t u\ dx \nonumber\\
&=&-\frac{\delta}{2}\ \frac{\mathrm{d}}{\mathrm{d}t}||\partial_xu||_2^2-\int_{-1}^1(\partial_xu\ \varphi+ u\ \partial_x \varphi)\ \partial_tu\ dx+r\ \int_{-1}^1  (u-u^2)\ \partial_t u\ dx.\nonumber
 \end{eqnarray}
Using Young and Cauchy-Schwarz inequalities we obtain
 \begin{eqnarray}
||\partial_t u||_2^2 &\leq &-\frac{\delta}{2}\ \frac{\mathrm{d}}{\mathrm{d}t}||\partial_xu||_2^2+\int_{-1}^1 |\partial_x u|^2\ |\varphi|^2\ dx+\frac{1}{4}||\partial_tu||_2^2
 +\int_{-1}^1  |u|^2\ |\partial_x\varphi|^2\ dx+\frac{1}{4} ||\partial_tu||_2^2\nonumber\\
 &+&r\ \frac{\mathrm{d}}{\mathrm{d}t}\int_{-1}^1  \left( \frac{u^2}{2}-\frac{u^3}{3}\right)\ dx,\nonumber
 \end{eqnarray}
which gives
\begin{equation*}\label{n8}\frac{\mathrm{d}}{\mathrm{d}t}\left( \frac{\delta}{2}\ ||\partial_xu||_2^2+\int_{-1}^1 F(u)\ dx\right)+\frac{1}{2} ||\partial_tu||_2^2 \leq  ||\partial_x u||_2^2\ ||\varphi||_\infty^2+||u||_\infty^2\ ||\partial_x\varphi||_2^2,
\end{equation*}
where $F(u)= r\left(- \frac{u^2}{2}+\frac{u^3}{3} \right)\geq -\frac{r}{6} $.\\
 Next we integrate the above inequality in time, and use \eqref{infii}, \eqref{dxsq} and Lemma \ref{lem} to obtain
 \begin{equation*}\label{n9}
-\frac{r}{3}+\frac{1}{2}\int_0^t ||\partial_tu||_2^2\ ds \leq C+C \int_0^t \left( ||\varphi||_\infty^2+||\partial_x\varphi||_2^2\right)\ ds\leq C
 \end{equation*}
for $t\geq0$ where $C$ is independent of $t$. We have thus proved \eqref{dtu}.
 \end{proof}

 To end the proof of Theorem \ref{theo2}, our aim now is to look at the large time behaviour of the solution.
\begin{proof}[Proof of Theorem \ref{theo2},\ (large time behaviour)] In this proof, we follow \cite{stabi}.\\

By Lemma \ref{lem}, the family $\lbrace u(t),\ t\geq 0\rbrace$ is bounded in $ W^{1,2}(-1,1)$. Since the embedding of $W^{1,2}(-1,1)$ in $L^2(-1,1)$ is compact then, there are a sequence of positive time $(t_n)$, such that $t_n\rightarrow \infty$, and $z\in L^2(-1,1)$ such that 
$$z=\lim\limits_{n\to \infty}u( t_n) \ \mathrm{in}\ L^2(-1,1) \ \mathrm{and\ a.e.\ in}\ (-1,1).$$ Consider
$$U_n(s,x)=u(t_n+s, x),\ \ x\in (-1,1),\ -1<s<1,\ n>0,$$
and
$$\Phi_n(s,x)=\varphi(t_n+s, x),\ \ -1<s<1.$$
We first prove that
\begin{equation}\label{eq1}U_n\longrightarrow z\ \ \mathrm{ as}\ n\rightarrow  \infty,\ \ \mathrm{ in}\  C \left([-1,1]; L^2 (-1,1)\right).\end{equation}
Indeed for each $s\in (-1,1)$
$$\int_{-1}^1|u(t_n+s, x)-u(t_n, x)|^2\ dx\leq \int_{-1}^1\int_{t_n-1}^{t_n+1}|\partial_t u|^2\ dt\ dx.$$
Hence
$$\sup_{s\in[-1,1]}||U_n(s)-u(t_n)||_{2}\leq \left[ 2\int_{-1}^1\int_{t_n-1}^\infty|\partial_t u|^2\ dt\ dx\right]^{\frac{1}{2}}. $$
The right hand side goes to zero as $n\rightarrow \infty$ by Lemma \ref{lem2}. Letting $n\rightarrow \infty$ in the above inequality gives \eqref{eq1}.\\
 
Next, using the definition of $D(u,\varphi)$ which is given in \eqref{Du} we obtain that
 \begin{eqnarray}
 &&\int_{-1}^1 \left(r\ ||U_n\ \log U_n\ (U_n-1)||_1+||\partial_x \sqrt {U_n} ||_2^2+||\Phi_n ||_2^2+\varepsilon\ ||\partial_x\Phi_n ||_2^2\right)\ ds\nonumber\\
 & &\leq \int_{t_n-1}^{t_n+1}\left(r\ ||u(s)\ \log u(s)\ (u(s)-1)||_1+||\partial_x \sqrt {u(s)}||_2^2+|| \varphi(s)||_2^2+\varepsilon \ ||\partial_x \varphi(s)||_2^2\right)\ ds\nonumber\\
 &&\leq 2\int_{t_n-1}^\infty D(u, \varphi)\ ds.\label{1}
 \end{eqnarray}
 The right-hand side of \eqref{1} goes to zero as $n\rightarrow \infty$ by \eqref{L1D}, so that
   $$\Phi_n\longrightarrow 0\  \ \mathrm{as}\ n\rightarrow \infty,\ \ \mathrm{in}\ L^2\left((-1,1); W^{1,2}(-1,1)\right). $$
   In addition, using Cauchy-Schwarz inequality, \eqref{czero} and \eqref{1} we obtain
   \begin{eqnarray}
   \int_{-1}^1  ||\partial_x U_n(s)||_1^2\ ds&=& 4\int_{-1}^1\left( \int_{-1}^1 \sqrt {U_n(s)}\ |\partial_x\sqrt {U_n(s)}|\ dx\right)^2\ ds\nonumber\\
   &\leq& 4\ \int_{-1}^1||U_n(s)||_1\ ||\partial_x \sqrt {U_n(s)}||_2^2\ ds\nonumber\\
   &\leq& C\int_{t_n-1}^\infty D(u, \varphi)\ ds.\nonumber
   \end{eqnarray}
   Since the right-hand side goes to zero as $n\rightarrow \infty$ by \eqref{L1D}, then we have
   \begin{equation}\label{eq2}\partial_x U_n\longrightarrow 0\  \ \mathrm{as}\ n\rightarrow \infty,\ \ \mathrm{in}\ L^2\left((-1,1); L^1(-1,1)\right). \end{equation}
Since the limit in the sense of distribution is unique, \eqref{eq1} and \eqref{eq2} yield that \begin{equation}\label{dxz}\partial_x z=0.\end{equation}

If $r=0$, \eqref{dxz} together with \eqref{norm1} and \eqref{eq1} give that $z=<u_0>$. We have thus shown that $<u_0>$ is the only cluster point of $\lbrace u(t),\ t\geq 0\rbrace$. Since $\lbrace u(t),\ t\geq 0\rbrace$ is relatively compact in $L^2(-1,1)$ thanks to its boundedness in $W^{1,2}(-1,1)$ (see Lemma \ref{lem}), we conclude that $u(t)$ converges to $<u_0>$ in $L^2(-1,1)$ as $t\longrightarrow \infty$.\\

If $r>0$, by \eqref{1} we have
\begin{equation}\label{n7}
\int_{-1}^1  ||U_n\ \log U_n\ (U_n-1)||_1\ ds\longrightarrow 0,\ \ \mathrm{as}\ n\rightarrow \infty,
\end{equation}
Since $(U_n)$ is bounded in $L^\infty\left((-1,1)\times (-1,1)\right)$ thanks to the boundness of $\left\{ u(t),\ t\geq0 \right\}$ in $W^{1,2}(-1,1)$ and the embedding of $W^{1,2}(-1,1)$ in $L^\infty(-1,1)$, we infer from \eqref{eq1}, \eqref{dxz}, \eqref{n7} that $z\ \log z\ (z-1)=0$, that is $z=0$ or $z=1$. Therefore $0$ and $1$ are the only two cluster points of $\lbrace u(t),\ t\geq 0\rbrace$ as $t \rightarrow \infty$. Since the $\omega$-limit set of $u$ is a compact connected subset of $L^2(-1,1)$, see \cite[Theorem 9.1.8]{anintroduction} for instance, we conclude that $u(t)$ converges either to $0$ or to $1$ in $L^2(-1,1)$ as $t\longrightarrow \infty$.

\end{proof}
\section{Limiting behaviour as $\varepsilon\rightarrow 0$}
When $E(u)=1-u$, letting $\varepsilon\rightarrow 0$ in \eqref{grin2} formally leads to \eqref{paraboliq} which is well-posed since $E'<0$ and $\delta>0$. The purpose of this section is to justify rigorously this fact and prove Theorem \ref{limitep}. Let $T>0$, $\delta>0$, $r\geq 0$, $\varepsilon>0$ and a nonnegative initial condition $u_0\in W^{1,2}(-1,1)$. We discuss the limit as $\varepsilon\rightarrow 0$ of the unique solution $u_\varepsilon$ of
\begin{equation}
\label{epsil}
\left\{
\begin{array}{llll}
\displaystyle \partial_t u_\varepsilon&=&\delta \ \partial^2_{xx} u_\varepsilon-\partial_x(u_\varepsilon\ \varphi_\varepsilon)+r\ u_\varepsilon\ (1-u_\varepsilon)& \  \mathrm{in}\ (0,T)\times(-1,1),\\
\displaystyle u_\varepsilon(0,x)&=& u_0(x)& \  \mathrm{in}\ (-1,1),\\
\displaystyle \partial_xu_\varepsilon(t,\pm 1)&=&0& \  \mathrm{on}\ (0, T),
\end{array}
\right.
\end{equation}
given by Theorem \ref{theo2}, where $\varphi_\varepsilon$ is the unique solution of\\

$
\left\{
\begin{array}{llll}
\displaystyle -\varepsilon\ \partial^2_{xx} \varphi_\varepsilon+\varphi_\varepsilon&=&-\partial_x u_\varepsilon& \  \mathrm{in}\ (0, T)\times(-1,1),\\
\displaystyle \varphi_\varepsilon(t,\pm 1)&=&0& \  \mathrm{on}\ (0, T).
\end{array}
\right.
$
\subsection{Estimates}
\begin{lemma}\label{6bis}
There is $C_1(T)$ independent of $\varepsilon$ such that
\begin{equation}\label{5}
\int_0^T \left(\delta\ ||\partial_x \sqrt {u_\varepsilon}||_2^2+\varepsilon\ ||\partial_x\varphi_\varepsilon||_2^2+||\varphi_\varepsilon||_2^2\right)\ dt\leq C_1(T),
\end{equation}
\end{lemma}
\begin{proof}
By \eqref{hb} see (the proof of Lemma \ref{estimate}), we have
\begin{equation*}\label{eq}
\frac{\mathrm{d}}{\mathrm{d}t}\int_{-1}^1 u_\varepsilon\ \log u_\varepsilon\ dx+\varepsilon\ ||\partial_x\varphi_\varepsilon||_2^2+||\varphi_\varepsilon||_2^2\leq -4\ \delta\int_{-1}^1 |\partial_x \sqrt{u_\varepsilon}|^2\ dx+2\ r,
\end{equation*}
from which \eqref{5} follows by a time integration.
\end{proof}
Using Gagliardo-Nirenberg inequality \eqref{gagliardo} we obtain the following estimate: 
\begin{lemma}\label{gabis}
 For $2\leq p\leq 6$, there exists $C_2(T,p)$ independent of $\varepsilon$ such that
\begin{equation}\label{ga}
\int_0^T||u_\varepsilon||_{\frac{p}{2}}^{\frac{p}{2}}\ dt\leq C_2(T,p).
\end{equation}
\end{lemma}
\begin{proof}
For $t\in (0,T)$, thanks to \eqref{5}, we can use Gagliardo-Nirenberg inequality \eqref{gagliardo} on $\sqrt{u_\varepsilon}$ and we obtain for all $p\in [2, \infty)$
\begin{equation}\label{gagliard}
||\sqrt{u_\varepsilon(t)}||_p\leq C \ ||\sqrt{u_\varepsilon(t)}||_{W^{1,2}}^\theta\ ||\sqrt{u_\varepsilon(t)}||_2^{1- \theta},
\end{equation}
where $$\theta=\frac{p-2}{2\ p}.$$
Therefore
\begin{eqnarray}
||\sqrt{u_\varepsilon(t)}||_p^p&\leq& C\ ||\sqrt{u_\varepsilon(t)}||_{W^{1,2}}^{\frac{p-2}{2}}\ ||\sqrt{u_\varepsilon(t)}||_2^{\frac{p+2}{2}}\nonumber\\
||u_\varepsilon(t)||_{\frac{p}{2}}^{\frac{p}{2}}&\leq&C\ ||\sqrt{u_\varepsilon(t)}||_{W^{1,2}}^{\frac{p-2}{2}}\ ||u_\varepsilon(t)||_1^{\frac{p+2}{4}}.\nonumber
\end{eqnarray}
Since 
$$||u_\varepsilon(t)||_1\leq ||u_0||_1+2\ r\ T,$$
 and $\frac{p-2}{2}\leq 2$ for $2\leq p\leq 6$ , the estimate \eqref{ga} follows from \eqref{5} and the previous inequalities.
\end{proof}
\begin{lemma}\label{7bis}
There is $C_3(T)$ independent of $\varepsilon$ such that
\begin{equation}\label{7}
\int_0^T ||\partial_xu_\varepsilon||_{\frac{3}{2}}^{\frac{3}{2}}\ dt\leq C_3(T).
\end{equation}
\end{lemma}
\begin{proof}
 H\"older and Young inequalities together with \eqref{5} and \eqref{ga} with $p=6$ yield
\begin{eqnarray}
\int_0^T ||\partial_xu_\varepsilon||_{\frac{3}{2}}^{\frac{3}{2}}\ dt&\leq&2^{\frac{3}{2}}\ \int_0^T ||\sqrt {u_\varepsilon}\ \partial_x\sqrt {u_\varepsilon} ||_{\frac{3}{2}}^{\frac{3}{2}}\ dt\nonumber\\
&\leq& 2^{\frac{3}{2}}\ \int_0^T ||u_\varepsilon||_3^{\frac{3}{4}}\ ||\partial_x \sqrt {u_\varepsilon}||_2^{\frac{3}{2}}\ dt\nonumber\\
&\leq& C\ \int_0^T \left( ||u_\varepsilon||_3^3+||\partial_x\sqrt {u_\varepsilon}||_2^2\right)\ dt\leq C_3(T),\nonumber
\end{eqnarray}
which gives the result.
\end{proof}
\begin{lemma}\label{duality}
There is $C_4(T)$ independent of $\varepsilon$ such that
$$\int_0^T ||\partial_t u_\varepsilon||_{(W^{2, \frac{3}{2}})'}\ dt\leq C_4(T).$$
\end{lemma}
\begin{proof}
Consider $\psi \in W^{2,\frac{3}{2}}(-1,1)$ and $t\in (0, T)$, we have
\begin{eqnarray}
&&\left\lvert\int_{-1}^1 \partial_t u_\varepsilon\ \psi\ dx\right\lvert\nonumber\\
&=&\left\lvert\int_{-1}^1\left[ \partial_x\left(\delta\ \partial_x u_\varepsilon-u_\varepsilon\ \varphi_\varepsilon\right)+r\ u_\varepsilon\ E(u_\varepsilon)\right]\ \psi\ dx\right\lvert\nonumber\\
&=& \left\lvert\int_{-1}^1 \left( -\delta\ \partial_x u_\varepsilon\ \partial_x\psi+u_\varepsilon\ \varphi_\varepsilon\ \partial_x\psi +r\ u_\varepsilon\ (1-u_\varepsilon)\ \psi\right)\ dx\right\lvert\nonumber\\
&\leq& \delta\ ||\partial_x\psi||_\infty\ ||\partial_x u_\varepsilon||_1+||\partial_x\psi||_\infty\ \ ||u_\varepsilon||_2\ ||\varphi_\varepsilon||_2+r\ ||u_\varepsilon(1-u_\varepsilon)||_1\ ||\psi||_\infty\nonumber.
\end{eqnarray} 
Using the embedding of $W^{2,\frac{3}{2}}(-1,1)$ in $W^{1,\infty}(-1,1)$, and Young's inequality, we end up with
\begin{eqnarray}
\left\lvert\int_{-1}^1\partial_t u_\varepsilon\ \psi\ dx\right\lvert
&\leq& \left(\delta\ ||\partial_x u_\varepsilon||_1+||u_\varepsilon||_2\ ||\varphi_\varepsilon||_2+r\ ||u_\varepsilon||_1+r\ ||u_\varepsilon||_2^2\right) \ ||\psi||_{W^{2,\frac{3}{2}}}\nonumber\\
&\leq& C \ \left (||\partial_x u_\varepsilon||_{\frac{3}{2}}+||u_\varepsilon||_2^2+||\varphi_\varepsilon||_2^2+1\right) \ ||\psi||_{W^{2,\frac{3}{2}}},\nonumber
\end{eqnarray}
and a duality argument gives
$$||\partial_t u_\varepsilon (t)||_{(W^{2,\frac{3}{2}})'}\leq C\ \left( ||\partial_x u_\varepsilon||_{\frac{3}{2}}+||u_\varepsilon||_2^2+||\varphi_\varepsilon||_2^2+1\right).$$
Integrating the above inequality over $(0,T)$ and using Young's inequality we obtain
$$\int_0^T||\partial_t u_\varepsilon (t)||_{(W^{2,\frac{3}{2}})'}\ dt\leq C(T)\ \int_0^T\left( ||\partial_x u_\varepsilon||_{\frac{3}{2}}^{\frac{3}{2}}+||u_\varepsilon||_2^2+||\varphi_\varepsilon||_2^2+1\right)\ dt.$$
By Lemma \ref{6bis}, Lemma \ref{7bis} and Lemma \ref{gabis} with $p=4$ the right-hand side of the above inequality is bounded independently of $\varepsilon$ and the proof of Lemma \ref{duality} is complete.
\end{proof}
\subsection{Convergence}
In this section we discuss the limit of $u_\varepsilon$ as $\varepsilon\rightarrow 0$. For that purpose, we study the compactness properties of $(u_\varepsilon, \varphi_\varepsilon)$.
\begin{proof}[Proof of Theorem \ref{limitep}]
Thanks to Lemma \ref{gabis} and Lemma \ref{7bis}, $(u_\varepsilon)_\varepsilon$ is bounded in\\
 $L^{\frac{3}{2}}((0,T); W^{1, \frac{3}{2}}(-1,1))$ while $(\partial_t u_\varepsilon)_\varepsilon$ is bounded in $L^1((0,T); (W^{2,\frac{3}{2}})'(-1,1))$\\ by Lemma \ref{duality}.
Since $W^{1, \frac{3}{2}}(-1,1)$ is compactly embedded in $C([-1,1])$ and $C([-1,1])$ is continuously embedded in $(W^{2,\frac{3}{2}})'(-1,1)$, it follows from \cite[Corollary 4]{compact} that $(u_\varepsilon)_\varepsilon$ is relatively compact in  $L^{\frac{3}{2}}\left((0,T); C([-1,1])\right)$. Therefore, there are a sequence $(\varepsilon_j)$ of positive real numbers, $\varepsilon_j\rightarrow 0$, and $u\in L^{\frac{3}{2}}\left( (0,T); W^{1,\frac{3}{2}}\right)$ such that 
\begin{equation}\label{n11}u_{\varepsilon_j}\rightharpoonup u\ \ \mathrm{in}\ L^{\frac{3}{2}}\left((0,T); W^{1,\frac{3}{2}}(-1,1)\right), \end{equation}
and 
\begin{equation}\label{n12}u_{\varepsilon_j}\rightarrow u\ \mathrm{in}\ L^{\frac{3}{2}}\left((0,T); C[-1,1]\right)\ \mathrm{and\ a.e.\ in}\ (0,T)\times (-1,1).\end{equation}
Since $(u_\varepsilon)_\varepsilon$ is bounded in $L^\infty\left((0,T);L^1 (-1,1)\right)$ by \eqref{czero}, it follows from \eqref{n12} 
$$\int_0^T||u_{\varepsilon_j}-u||_2^{3}\ dt\leq  \int_0^T ||u_{\varepsilon_j}-u||_1^{\frac{3}{2}}\ ||u_{\varepsilon_j}-u||_\infty^{\frac{3}{2}}\ dt\leq  \int_0^T ||u_{\varepsilon_j}-u||_\infty^{\frac{3}{2}}\ dt\rightarrow 0,$$ when $\varepsilon_j\rightarrow 0$. In particular, we have
\begin{equation} \label{l2U}u_{\varepsilon_j} \longrightarrow u, \ \mathrm{in}\ L^2\left((0,T)\times (-1,1)\right).\end{equation}
Observe that the nonnegativity of $u$ follows easily from that of $(u_{\varepsilon_j})$ by \eqref{l2U}.\\

 Owing to Lemma \ref{6bis}, we may also assume that
\begin{eqnarray}\label{conf}\varphi_{\varepsilon_j}\rightharpoonup \varphi\ \mathrm{in}\ L^2((0,T)\times(-1,1))\ \ \mathrm{as}\ \varepsilon_j\rightarrow 0,\\
\varepsilon_j\ \partial_x\varphi_{\varepsilon_j}\rightarrow 0\ \mathrm{in}\ L^2((0,T)\times(-1,1))\ \ \mathrm{as}\ \varepsilon_j\rightarrow 0.\label{n13}
\end{eqnarray}
It remains to identify the equations solved by the limit $u$ of $(u_{\varepsilon_j})$. Let $\psi\in C^2([0,T]\times[-1,1])$ with $\psi(T)=0$. Since
\begin{equation}\label{n14}
\int_0^T \int_{-1}^1\partial_t u_{\varepsilon_j}\  \psi\ dxdt=\int_0^T\int_{-1}^1 \left((-\delta\ \partial_x u_{\varepsilon_j}+u_{\varepsilon_j}\ \varphi_{\varepsilon_j})\ \partial_x \psi+r\ u_{\varepsilon_j}\ E(u_{\varepsilon_j})\ \psi\right)\ \ dxdt
\end{equation}
and
\begin{equation}\label{n15}
\varepsilon_j \int_0^T\int_{-1}^1 \partial_x \varphi_{\varepsilon_j}\ \partial_x\psi\  dxdt+\int_0^T\int_{-1}^1\varphi_{\varepsilon_j}\ \psi\ \ dxdt=-\int_0^T\int_{-1}^1 \partial_x u_{\varepsilon_j}\ \psi \ dxdt.
\end{equation}
Owing to \eqref{n11}, \eqref{conf} and \eqref{n13}, it is straightforward to pass to the limit as $\varepsilon_j\rightarrow 0$ in \eqref{n15} and find 
$$\int_0^T\int_{-1}^1\varphi\ \psi\ dxdt=-\int_0^T\int_{-1}^1 \partial_x u\ \psi\  dxdt,$$ which gives that \begin{equation}\label{res}\varphi=-\partial_xu.\end{equation}
Next, by \eqref{n12} and \eqref{n11} we see that
$$\int_0^T \int_{-1}^1\partial_t u_{\varepsilon_j}\ \psi\ dxdt\longrightarrow -\int_0^T\int_{-1}^1  u\ \partial_t\psi\ dxdt-\int_{-1}^1  u_0(x)\ \psi(0,x)\ dx, \ \ \mathrm{as}\ \varepsilon_j\rightarrow 0,$$
and
$$\int_0^T\int_{-1}^1 \partial_x u_{\varepsilon_j}\ \partial_x\psi \  dxdt\longrightarrow \int_0^T\int_{-1}^1 \partial_x u\ \partial_x\psi\ dxdt\ \ \mathrm{as}\ \varepsilon_j\rightarrow 0.$$
From \eqref{l2U}, \eqref{conf} and \eqref{res} we see that
$$\int_0^T\int_{-1}^1 u_{\varepsilon_j}\ \varphi_{\varepsilon_j}\ \partial_x\psi\ dxdt\longrightarrow -\int_0^T\int_{-1}^1 u\ \partial_xu\ \partial_x\psi\ dxdt, \ \ \mathrm{as}\ \varepsilon_j\rightarrow 0.$$
From \eqref{l2U} we get
$$r\int_0^T\int_{-1}^1 u_{\varepsilon_j}\ E(u_{\varepsilon_j})\ \psi\ dxdt\longrightarrow r\ \int_0^T\int_{-1}^1 u\ E(u)\ \psi\ dxdt, \ \ \mathrm{as}\ \varepsilon_j\rightarrow 0.$$
Thus we conclude that $u$ satisfies
\begin{equation*}\int_0^T \langle\partial_t u, \psi\rangle \ dt=\int_0^T\int_{-1}^1 \left((-\delta\ \partial_x u-u\ \partial_xu)\ \partial_x \psi+r\ u\ E(u)\ \psi\right)\ dx\ dt,\end{equation*}
for all test functions $\psi$. Therefore, $u$ is a weak solution of \eqref{limite}, and classical regularity results ensure that $u$ is actually a classical solution of \eqref{limite}. Since it is unique and the only possible cluster point of $(u_\varepsilon)_\varepsilon$ in $L^2\left( (0,T)\times (-1,1)\right)$, we conclude that the whole family  $(u_\varepsilon)_\varepsilon$ converges to $u$ in $L^2 \left( (0,T)\times (-1,1)\right)$ as $\varepsilon\rightarrow 0$.
\end{proof}
\section*{Acknowledgment}
I thank Philippe Lauren\c cot for his helpful advices and comments during this work.

\end{document}